\documentclass[12pt,twoside]{amsart}
\usepackage[latin1]{inputenc}
\usepackage{amsmath, amsthm, amscd, amsfonts, amssymb, graphicx}
\usepackage[hidelinks]{hyperref}
\usepackage{titletoc}

\newtheorem{thm}{Theorem}[section]

\newtheorem{lem}[thm]{Lemma}

\newtheorem{defn}[thm]{Definition}

\numberwithin{equation}{section}

\begin{document}

\title{{\bf Elliptic genera and $E_{8}$ Bundles in odd dimensions }}

\author{ Siyao Liu \ \ Yong Wang* \\
 }

\date{}

\thanks{{\scriptsize
\hskip -0.4 true cm \textit{2010 Mathematics Subject Classification:}
58C20; 57R20; 53C80.
\newline \textit{* Corresponding author.}}}

\maketitle

\begin{abstract}
This paper aims to derive new anomaly cancellation formulas by combining modular forms with $E_8$ and $E_{8}\times E_{8}$ bundles.
To this end, we systematically twist and generalize known $SL(2,\mathbf{Z})$ modular forms to define new modular forms associated with these bundles on odd-dimensional spin and spin$^c$ manifolds, leading to a new series of anomaly cancellation formulas.
\end{abstract}

\vskip 0.2 true cm

\tableofcontents

\section{Introduction}

In \cite{AW}, Alvarez-Gaum\'{e} and Witten discovered the ``miraculous cancellation" formula for gravitational anomaly in a $12$-dimensional smooth Riemannian manifold $M$:
\begin{align}
\Big\{\widehat{L}(TM, \nabla^{TM})\Big\}^{(12)}=\Big\{\widehat{A}(TM, \nabla^{TM})ch(T_{\mathbf{C}}M, \nabla^{T_{\mathbf{C}}M})-32\widehat{A}(TM, \nabla^{TM})\Big\}^{(12)},\nonumber
\end{align}
where $T_{\mathbf{C}}M$ denotes the complexification of $TM$ and $\nabla^{T_{\mathbf{C}}M}$ is canonically induced from $\nabla^{TM},$ the Levi-Civita connection associated to the Riemannian structure of $M.$
This formula reveals the beautiful relationship between the top components of the Hirzebruch $\widehat{L}$-form and $\widehat{A}$-form.
Liu established higher-dimensional ``miraculous cancellation" formulas for $(8k+4)$-dimensional Riemannian manifolds through the development of modular invariance properties of characteristic forms \cite{L}.
In \cite{HZ1,HZ2}, Han and Zhang established general cancellation formulas for $(8k+4)$-dimensional smooth Riemannian manifolds that are each equipped with a complex line bundle, through the study of modular invariance properties of some characteristic forms.
Subsequently, further cancellation formulas for higher-dimensional Riemannian manifolds were obtained in \cite{HH}.
In \cite{W1}, Wang proved more general cancellation formulas for $(8k+2)$ and $(8k+6)$-dimensional smooth Riemannian manifolds.
Wang extended this approach to both spin and spin$^c$ manifolds, deriving corresponding cancellation formulas in each case \cite{W2}.
Using the Eisenstein series as a new methodological tool, Han, Liu and Zhang derived more new anomaly cancellation formulas in their research \cite{HLZ}.
In \cite{LW}, the authors generalized the Han-Zhang and Han-Liu-Zhang cancellation formulas to the $(a, b)$ type cancellation formulas for $(4k)$-dimensional Riemannian manifolds.
The study of odd-dimensional manifolds, which are generally more complex, presents additional challenges.
Addressing this complexity, Liu and Wang in \cite{LW1} and \cite{LW2} derived new anomaly cancellation formulas for $(4k-1)$-dimensional manifolds by studying  modular invariance properties.
In recent work, by using the $SL(2,\mathbf{Z})$ modular forms introduced in \cite{L2} and \cite{CHZ}, Guan constructed new modular forms over $SL(2,\mathbf{Z})$, $\Gamma^0(2)$ and $\Gamma_0(2)$ in odd dimensions and obtained new cancellation formulas for odd-dimensional spin and spin$^c$ manifolds, respectively \cite{GWL}.

The exploration of $E_8$ bundles is deeply motivated by their physical implications. 
As a key example, \cite{AE} considers an $E_8$ bundle over an $11$-fold.
Leveraging this, recent research has made concrete progress in anomaly cancellation by systematically incorporating $E_8$ bundles into the framework, leading to new derived formulas.
In \cite{HLZ2}, Han, Liu and Zhang demonstrated that both the Ho$\check{r}$ava-Witten anomaly factorization formula for the gauge group $E_{8}$ and the Green-Schwarz anomaly factorization formula for $E_{8}\times E_{8}$ can be derived using modular forms of weight $14.$
Subsequently, in \cite{HHLZ}, by $E_{8}$ bundles, Han, Huang, Liu and Zhang introduced a modular form of weight $14$ and a modular form of weight $10$ over $SL(2,\mathbf{Z})$ and obtained interesting anomaly cancellation formulas on $12$-dimensional manifolds.
In a related development, Wang and Yang \cite{WY} twisted the Chen-Han-Zhang $SL(2,\mathbf{Z})$ modular form \cite{CHZ} by $E_{8}$ bundles, yielding $SL(2,\mathbf{Z})$ modular forms of weights $14$ and $10$ for $14$ and $10$-dimensional spin manifolds, respectively.
Building on this, they applied the same technique to the modular form of Liu \cite{L}, thereby constructing $\Gamma^0(2)$ and $\Gamma_0(2)$ modular forms. These constructions led to novel anomaly cancellation formulas involving characteristic forms.

Building upon this background, this paper aims to explore the combination of modular forms with $E_8$ bundles, with the goal of deriving corresponding new anomaly cancellation formulas.
Specifically, starting from the $SL(2,\mathbf{Z})$ modular forms established in \cite{L2}, \cite{CHZ}, and \cite{GWL}, we twist and generalize them using $E_{8}$ and $E_{8}\times E_{8}$ bundles.
This approach enables the systematic construction of new modular forms on odd-dimensional spin and spin$^c$ manifolds, leading to new anomaly cancellation formulas.

This paper is organized as follows:
Section 2 provides the necessary definitions and basic notions.
Section 3 focuses on $E_{8}$ bundles. We combine $E_{8}$ bundles with modular forms to construct new modular forms over $SL(2,\mathbf{Z})$ on odd-dimensional spin and spin$^c$ manifolds, thereby deriving new anomaly cancellation formulas.
Section 4 extends the study to $E_{8}\times E_{8}$ bundles. Following a similar methodology, we construct associated modular forms for these bundles on spin and spin$^c$ manifolds and systematically derive a new series of cancellation formulas.

\section{Characteristic Forms and Modular Forms}

Firstly, we give some definitions and basic notions on characteristic forms and modular forms that will be used throughout the paper.
For the details, see \cite{A,H,Z}.

\subsection{characteristic forms}\hfill\\

Let $M$ be a Riemannian manifold, $\nabla^{TM}$ be the associated Levi-Civita connection on $TM$ and $R^{TM}=(\nabla^{TM})^{2}$ be the curvature of $\nabla^{TM}.$
According to the detailed descriptions in \cite{Z}, let $\widehat{A}(TM, \nabla^{TM})$ and $\widehat{L}(TM, \nabla^{TM})$ be the Hirzebruch characteristic forms defined respectively by
\begin{align}
\widehat{A}(TM, \nabla^{TM})&=\det\nolimits^{\frac{1}{2}}\left(\frac{\frac{\sqrt{-1}}{4\pi}R^{TM}}{\sinh(\frac{\sqrt{-1}}{4\pi}R^{TM})}\right),\\
\widehat{L}(TM, \nabla^{TM})&=\det\nolimits^{\frac{1}{2}}\left(\frac{\frac{\sqrt{-1}}{2\pi}R^{TM}}{\tanh(\frac{\sqrt{-1}}{4\pi}R^{TM})}\right).
\end{align}

Let $E, F$ be two Hermitian vector bundles over $M$ carrying Hermitian connection $\nabla^{E}, \nabla^{F}$ respectively.
Let $R^{E}=(\nabla^{E})^{2}$ (resp. $R^{F}=(\nabla^{F})^{2}$) be the curvature of $\nabla^{E}$ (resp. $\nabla^{F}$).
If we set the formal difference $G=E-F,$ then $G$ carries an induced Hermitian connection $\nabla^{G}$ in an obvious sense.
We define the associated Chern character form as
\begin{align}
\text{ch}(G, \nabla^{G})=\text{tr}\left[\exp\left(\frac{\sqrt{-1}}{2\pi}R^{E}\right)\right]-\text{tr}\left[\exp\left(\frac{\sqrt{-1}}{2\pi}R^{F}\right)\right].
\end{align}

For any complex number $t,$ let
\begin{align}
\wedge_{t}(E)&=\mathbf{C}|_{M}+tE+t^{2}\wedge^{2}(E)+\cdot\cdot\cdot,\\
S_{t}(E)&=\mathbf{C}|_{M}+tE+t^{2}S^{2}(E)+\cdot\cdot\cdot,
\end{align}
denote respectively the total exterior and symmetric powers of $E,$ which live in $K(M)[[t]].$
The following relations between these operations hold,
\begin{align}
S_{t}(E)=\frac{1}{\wedge_{-t}(E)},~~~~ \wedge_{t}(E-F)=\frac{\wedge_{t}(E)}{\wedge_{t}(F)}.
\end{align}
Moreover, if $\{ \omega_{i} \}, \{ \omega'_{j} \}$ are formal Chern roots for Hermitian vector bundles $E, F$ respectively, then
\begin{align}
\text{ch}(\wedge_{t}(E))=\prod_{i}(1+e^{\omega_{i}}t).
\end{align}
Then we have the following formulas for Chern character forms,
\begin{align}
\text{ch}(S_{t}(E))=\frac{1}{\prod\limits_{i}(1-e^{\omega_{i}}t)},~~~ \text{ch}(\wedge_{t}(E-F))=\frac{\prod\limits_{i}(1+e^{\omega_{i}}t)}{\prod\limits_{j}(1+e^{\omega'_{j}}t)}.
\end{align}

If $W$ is a real Euclidean vector bundle over $M$ carrying a Euclidean connection $\nabla^{W},$ then its complexification $W_{\mathbf{C}}=W\otimes \mathbf{C}$ is a complex vector bundle over $M$ carrying a canonical induced Hermitian metric from that of $W,$ as well as a Hermitian connection $\nabla^{W_{\mathbf{C}}}$ induced from $\nabla^{W}.$
If $E$ is a vector bundle (complex or real) over $M,$ set $\widetilde{E}=E-\dim E$ in $K(M)$ or $KO(M).$

\subsection{Some properties about the Jacobi theta functions and modular forms}\hfill\\

Refer to \cite{C}, we recall the four Jacobi theta functions are defined as follows:
\begin{align}
&\theta(v, \tau)=2q^{\frac{1}{8}}\sin(\pi v)\prod^{\infty}_{j=1}[(1-q^{j})(1-e^{2\pi\sqrt{-1}v}q^{j})(1-e^{-2\pi\sqrt{-1}v}q^{j})];\\
&\theta_{1}(v, \tau)=2q^{\frac{1}{8}}\cos(\pi v)\prod^{\infty}_{j=1}[(1-q^{j})(1+e^{2\pi\sqrt{-1}v}q^{j})(1+e^{-2\pi\sqrt{-1}v}q^{j})];\\
&\theta_{2}(v, \tau)=\prod^{\infty}_{j=1}[(1-q^{j})(1-e^{2\pi\sqrt{-1}v}q^{j-\frac{1}{2}})(1-e^{-2\pi\sqrt{-1}v}q^{j-\frac{1}{2}})];\\
&\theta_{3}(v, \tau)=\prod^{\infty}_{j=1}[(1-q^{j})(1+e^{2\pi\sqrt{-1}v}q^{j-\frac{1}{2}})(1+e^{-2\pi\sqrt{-1}v}q^{j-\frac{1}{2}})],
\end{align}
where $q=e^{2\pi\sqrt{-1}\tau}$ with $\tau\in\mathbf{H},$ the upper half complex plane.

Let
\begin{align}
\theta'(0, \tau)=\frac{\partial \theta(v, \tau)}{\partial v}\Big|_{v=0}.
\end{align}
Then the following Jacobi identity holds,
\begin{align}
\theta'(0, \tau)=\pi\theta_{1}(0, \tau)\theta_{2}(0, \tau)\theta_{3}(0, \tau).
\end{align}

In what follows,
\begin{align}
SL(2,{\bf{Z}})=\Big\{ \Big( \begin{array}{cc}
a & b\\
c & d
\end{array}
\Big)\Big|a, b, c, d\in\mathbf{Z}, ad-bc=1
\Big \}
\end{align}
stands for the modular group.
Write $S=\Big(\begin{array}{cc}
0 & -1\\
1 & 0
\end{array}\Big),
T=\Big(\begin{array}{cc}
1 & 1\\
0 & 1
\end{array}\Big)$
be the two generators of $SL(2,{\bf{Z}}).$
They act on $\mathbf{H}$ by $S\tau=-\frac{1}{\tau}, T\tau=\tau+1.$
One has the following transformation laws of theta functions under the actions of $S$ and $T:$
\begin{align}
&\theta(v, \tau+1)=e^{\frac{\pi\sqrt{-1}}{4}}\theta(v, \tau),~~~~ \theta(v, -\frac{1}{\tau})=\frac{1}{\sqrt{-1}}\Big(\frac{\tau}{\sqrt{-1}}\Big)^{\frac{1}{2}}e^{\pi\sqrt{-1}\tau v^{2}}\theta(\tau v, \tau);\\
&\theta_{1}(v, \tau+1)=e^{\frac{\pi\sqrt{-1}}{4}}\theta_{1}(v, \tau),~~~~ \theta_{1}(v, -\frac{1}{\tau})=\Big(\frac{\tau}{\sqrt{-1}}\Big)^{\frac{1}{2}}e^{\pi\sqrt{-1}\tau v^{2}}\theta_{2}(\tau v, \tau);\\
&\theta_{2}(v, \tau+1)=\theta_{3}(v, \tau),~~~~ \theta_{2}(v, -\frac{1}{\tau})=\Big(\frac{\tau}{\sqrt{-1}}\Big)^{\frac{1}{2}}e^{\pi\sqrt{-1}\tau v^{2}}\theta_{1}(\tau v, \tau);\\
&\theta_{3}(v, \tau+1)=\theta_{2}(v, \tau),~~~~ \theta_{3}(v, -\frac{1}{\tau})=\Big(\frac{\tau}{\sqrt{-1}}\Big)^{\frac{1}{2}}e^{\pi\sqrt{-1}\tau v^{2}}\theta_{3}(\tau v, \tau).
\end{align}
Differentiating the above transformation formulas, we get that
\begin{align}
&\theta'(v, \tau+1)=e^{\frac{\pi\sqrt{-1}}{4}}\theta'(v, \tau),\\
&\theta'(v, -\frac{1}{\tau})=\frac{1}{\sqrt{-1}}\Big(\frac{\tau}{\sqrt{-1}}\Big)^{\frac{1}{2}}e^{\pi\sqrt{-1}\tau v^{2}}(2\pi\sqrt{-1}\tau v\theta(\tau v, \tau)+\tau\theta'(\tau v, \tau));\nonumber\\
&\theta_{1}'(v, \tau+1)=e^{\frac{\pi\sqrt{-1}}{4}}\theta_{1}'(v, \tau),\\
&\theta_{1}'(v, -\frac{1}{\tau})=\Big(\frac{\tau}{\sqrt{-1}}\Big)^{\frac{1}{2}}e^{\pi\sqrt{-1}\tau v^{2}}(2\pi\sqrt{-1}\tau v\theta_{2}(\tau v, \tau)+\tau\theta'_{2}(\tau v, \tau));\nonumber\\
&\theta'_{2}(v, \tau+1)=\theta'_{3}(v, \tau),\\
&\theta'_{2}(v, -\frac{1}{\tau})=\Big(\frac{\tau}{\sqrt{-1}}\Big)^{\frac{1}{2}}e^{\pi\sqrt{-1}\tau v^{2}}(2\pi\sqrt{-1}\tau v\theta_{1}(\tau v, \tau)+\tau\theta'_{1}(\tau v, \tau));\nonumber\\
&\theta'_{3}(v, \tau+1)=\theta'_{2}(v, \tau),\\
&\theta'_{3}(v, -\frac{1}{\tau})=\Big(\frac{\tau}{\sqrt{-1}}\Big)^{\frac{1}{2}}e^{\pi\sqrt{-1}\tau v^{2}}(2\pi\sqrt{-1}\tau v\theta_{3}(\tau v, \tau)+\tau\theta'_{3}(\tau v, \tau)).\nonumber
\end{align}
Therefore
\begin{align}
\theta'(0, -\frac{1}{\tau})=\frac{1}{\sqrt{-1}}\Big(\frac{\tau}{\sqrt{-1}}\Big)^{\frac{1}{2}}\tau\theta'(0, \tau).
\end{align}

\begin{defn}
A modular form over $\Gamma,$ a subgroup of $SL(2,{\bf{Z}}),$ is a holomorphic function $f(\tau)$ on $\mathbf{H}$ such that
\begin{align}
f(g\tau):=f\Big(\frac{a\tau+b}{c\tau+d}\Big)=\chi(g)(c\tau+d)^{k}f(\tau),~\forall g=\Big(\begin{array}{cc}
a & b\\
c & d
\end{array}\Big)\in\Gamma,
\end{align}
where $\chi:\Gamma\rightarrow\mathbf{C}^{*}$ is a character of $\Gamma,$ $k$ is called the weight of $f.$
\end{defn}

Let $E_{2}(\tau)$ be Eisenstein series which is a quasimodular form over $SL(2,{\bf{Z}}),$ satisfying
\begin{align}
E_{2}\Big(\frac{a\tau+b}{c\tau+d}\Big)=(c\tau+d)^{2}E_{2}(\tau)-\frac{6\sqrt{-1}c(c\tau+d)}{\pi}.
\end{align}
In particular, we have
\begin{align}
&E_{2}(\tau+1)=E_{2}(\tau),\\
&E_{2}\Big(-\frac{1}{\tau}\Big)=\tau^{2}E_{2}(\tau)-\frac{6\sqrt{-1}\tau}{\pi},
\end{align}
and
\begin{align}
E_{2}(\tau)=1-24q-72q^{2}+\cdot\cdot\cdot,
\end{align}
where the $``\cdot\cdot\cdot"$ terms are the higher degree terms, all of which have integral coefficients.

For the principal $E_{8}$ bundles $P_{i}, i=1,2$ consider the associated bundles
\begin{align}
\mathcal{V}_{i}=\sum_{k=0}^{\infty}(P_{i}\times_{\rho_{k}}V_{k})q^{k}\in K(Z)[[q]].
\end{align}
Let $W_{i}=P_{i}\times_{\rho_{1}}V_{1}, i=1,2$ be the complex vector bundles associated to the adjoint representation $\rho_{1},$ $\overline{W_{i}}=P_{i}\times_{\rho_{2}}V_{2}, i=1,2$ be the complex vector bundles associated to the adjoint representation $\rho_{2}.$

Following the notation of \cite{HLZ2}, we have that there are formal two forms $y_{l}^{i}, 1\leq l\leq 8, i=1,2$ such that
\begin{align}
\varphi(\tau)^{(8)}ch(\mathcal{V}_{i})=\frac{1}{2}\Big(\prod_{l=1}^{8}\theta_{1}(y_{l}^{i}, \tau)+\prod_{l=1}^{8}\theta_{2}(y_{l}^{i}, \tau)+\prod_{l=1}^{8}\theta_{3}(y_{l}^{i}, \tau)\Big),
\end{align}
and
\begin{align}
\sum_{l=1}^{8}(2\pi\sqrt{-1}y_{l}^{i})^{2}=-\frac{1}{30}c_{2}(W_{i}),
\end{align}
where $\varphi(\tau)=\prod_{n=1}^{\infty}(1-q^{n}),$ $c_{2}(W_{i})$ denotes the second Chern class of $W_{i}.$

\section{Modular forms and Witten genus associated with $E_{8}$ bundles}

In this section, we construct modular forms over $SL(2,{\bf{Z}})$ associated with $E_{8}$ bundles.
Furthermore, we derive new cancellation formulas for these modular forms on odd-dimensional spin and spin$^c$ manifolds, respectively.

\subsection{Modular forms in $(4k-1)$-dimensional spin manifolds}\hfill\\

Let $M$ be a $(4k-1)$-dimensional spin manifold and $\triangle(M)$ be the spinor bundle.
Suppose that
\begin{equation}
   \Theta_1(T_{\mathbf{C}}M)=
   \bigotimes _{n=1}^{\infty}S_{q^n}(\widetilde{T_{\mathbf{C}}M})\otimes
\bigotimes _{m=1}^{\infty}\wedge_{q^m}(\widetilde{T_{\mathbf{C}}M})
,\end{equation}
\begin{equation}
\Theta_2(T_{\mathbf{C}}M)=\bigotimes _{n=1}^{\infty}S_{q^n}(\widetilde{T_{\mathbf{C}}M})\otimes
\bigotimes _{m=1}^{\infty}\wedge_{-q^{m-\frac{1}{2}}}(\widetilde{T_{\mathbf{C}}M}),
\end{equation}
\begin{equation}
\Theta_3(T_{\mathbf{C}}M)=\bigotimes _{n=1}^{\infty}S_{q^n}(\widetilde{T_{\mathbf{C}}M})\otimes
\bigotimes _{m=1}^{\infty}\wedge_{q^{m-\frac{1}{2}}}(\widetilde{T_{\mathbf{C}}M}).
\end{equation}

We recall the odd Chern character associated to a smooth map g from M to the general linear group $GL(N,\mathbf{C}),$
where $N$ is a positive integer (see \cite{LW1,GWL}).
Let $d$ denote a trivial connection on $\mathbf{C}^{N}|_{M}$.
The associated cohomology class is denoted by $c_g(M,[g])$, which corresponds to the closed $n$-form
\begin{equation}
  c_n(\mathbf{C}^{N}|_M,g,d)=\left(\frac{1}{2\pi\sqrt{-1}}\right)^{\frac{n+1}{2}}\mathrm{Tr}[(g^{-1}dg)^n].
\end{equation}
The odd Chern character form ${\rm ch}(\mathbf{C}^{N}|_M,g,d)$ associated to $g$ and $d$ by definition is
\begin{equation}
  {\rm ch}(\mathbf{C}^{N}|_M,g,d)=\sum^{\infty}_{n=1}\frac{n!}{(2n+1)!}c_{2n+1}((\mathbf{C}^{N}|_M,g,d)).
\end{equation}
Let the connection $\nabla_{u}$ on the trivial bundle $\mathbf{C}^{N}|_M$ defined by
\begin{equation}
  \nabla_u=(1-u)d+ug^{-1}\cdot d \cdot g,\ \ u\in[0,1].
\end{equation}
Then we have
\begin{equation}
  d{\rm ch}(\mathbf{C}^{N}|_M,g,d)={\rm ch}(\mathbf{C}^{N}|_M,d)-{\rm ch}(\mathbf{C}^{N}|_M,g^{-1}\cdot d\cdot g).
\end{equation}

Now let $g:M\to SO(N)$ and we assume that $N$ is even and large enough.
Let $E$ denote the trivial real vector bundle of rank $N$ over $M$. We equip $E$ with the canonical trivial metric and trivial connection $d$.
Set
\begin{equation}
\nabla_u=d+ug^{-1}dg,\ \ u\in[0,1].
\end{equation}
Let $R_u$ be the curvature of $\nabla_u$, then
\begin{equation}
  R_u=(u^2-u)(g^{-1}dg)^2.
\end{equation}

We also consider the complexification of $E$ and $g$ extends to a unitary automorphism of $E_{\mathbf{C}}$. The connection $\nabla_u$ extends to a Hermitian connection on $E_{\mathbf{C}}$ with curvature still given by (3.6). Let $\Delta(E)$ be the spinor bundle of $E$, which is a trivial Hermitian
bundle of rank $2^{\frac{N}{2}}$. We assume that $g$ has a lift to the spin group ${\rm Spin}(N):g^{\Delta}:M\to {\rm Spin}(N)$. So $g^{\Delta}$ can be viewed as an automorphism of $\Delta(E)$ preserving the Hermitian metric. We lift $d$ on $E$ to be a trivial Hermitian connection $d^{\Delta}$ on $\Delta(E)$, then
\begin{equation}
  \nabla_u^{\Delta}=(1-u)d^{\Delta}+u(g^{\Delta})^{-1}\cdot d^{\Delta} \cdot g^{\Delta},\ \ u\in[0,1]
\end{equation}
lifts $\nabla_u$ on $E$ to $\Delta(E)$.

Set $Q_j(E),j=1,2,3$ be the virtual bundles defined as follows:
\begin{align}
Q_1(E)&=\triangle(E)\otimes
   \bigotimes _{n=1}^{\infty}\wedge_{q^n}(\widetilde{E_C});\\
Q_2(E)&=\bigotimes _{n=1}^{\infty}\wedge_{-q^{n-\frac{1}{2}}}(\widetilde{E_C});\\
Q_3(E)&=\bigotimes _{n=1}^{\infty}\wedge_{q^{n-\frac{1}{2}}}(\widetilde{E_C}).
\end{align}
Let $g$ on $E$ have a lift $g^{Q(E)}$ on $Q(E)$ and $\nabla_u$ have a lift $\nabla^{Q(E)}_u$ on $Q(E)$. Following \cite{HY}, we defined ${\rm ch}(Q(E),g^{Q(E)},d,\tau)$ as following
\begin{align}
{\rm ch}(Q(E),\nabla^{Q(E)}_0,\tau)-{\rm ch}(Q(E),\nabla^{Q(E)}_1,\tau)=d{\rm ch}(Q(E),g^{Q(E)},d,\tau),
\end{align}
where
\begin{align}
Q(E)=Q_1(E)\otimes Q_2(E)\otimes Q_3(E),
\end{align}
and
\begin{align}
{\rm ch}(Q(E),g^{Q(E)},d,\tau)=-\frac{2^{\frac{N}{2}}}{8\pi^2}\int^1_0{\rm Tr}\left[g^{-1}dg\left(A\right)\right]du,
\end{align}
with
\begin{align}
A=\frac{\theta'_1(R_u/(4\pi^2),\tau)}{\theta_1(R_u/(4\pi^2),\tau)}
+\frac{\theta'_2(R_u/(4\pi^2),\tau)}{\theta_2(R_u/(4\pi^2),\tau)}+\frac{\theta'_3(R_u/(4\pi^2),\tau)}
{\theta_3(R_u/(4\pi^2),\tau)}.
\end{align}

\begin{lem}(\cite{GWL,HY})
If $c_3(E_C,g,d)=0,$ then for any integer $r\geq 1,$  we have
\begin{align}
{\rm ch}(Q(E),g^{Q(E)},d,\tau+1)^{(4r-1)}={\rm ch}(Q(E),g^{Q(E)},d,\tau)^{(4r-1)},
\end{align}
and
\begin{align}
{\rm ch}\Big(Q(E),g^{Q(E)},d,-\frac{1}{\tau}\Big)^{(4r-1)}=\tau^{2r}{\rm ch}(Q(E),g^{Q(E)},d,\tau)^{(4r-1)},
\end{align}
so ${\rm ch}(Q(E),g^{Q(E)},d,\tau)^{(4r-1)}$ are modular forms of weight $2r$ over $SL(2,\bf{Z})$.
\end{lem}

Let $g: M\to SO(N)$ and $N$ is even and large enough.
Let $E$ denote the trivial real vector bundle of rank $N$ over $M$.
Define a modular form on $\mathrm{SL}(2,\mathbf{Z})$ twisted by $E_8$ bundles as
\begin{align}
Q(\nabla^{TM},P_{i},g,d,\tau)=&\{e^{\frac{1}{24}E_{2}(\tau)\cdot\frac{1}{30}c_{2}(W_{i})} \widehat{A}(TM,\nabla^{TM}){\rm ch}([\triangle(M)\otimes \Theta_1(T_{C}M)\\
&+2^{2k-1}\Theta_2(T_{C}M)+2^{2k-1}\Theta_3(T_{C}M)]){\rm ch}(Q(E),g^{Q(E)},d,\tau)\nonumber\\
&\cdot\varphi(\tau)^{(8)}ch(\mathcal{V}_{i})\}^{(4k-1)}.\nonumber
\end{align}
Then we have the following theorem.

\begin{thm}
Let ${\rm dim}M=4k-1,$ for $k\leq 4.$ Suppose $g$ has a lift to the spin group ${\rm Spin}(N):g^{\Delta}:M\to {\rm Spin}(N)$. If $c_3(E,g,d)=0$, then $Q(\nabla^{TM},P_{i},g,d,\tau)$ is a modular form over $SL(2,{\bf Z})$ with the weight $2k+4$.
\end{thm}
\begin{proof}
Set
\begin{align}
Q(\nabla^{TM},P_{i},g,d,\tau)=&Q(M,P_{i},\tau)\cdot{\rm ch}(Q(E),g^{Q(E)},d,\tau)^{(4r-1)},
\end{align}
from this
\begin{align}
Q(M,P_{i},\tau)=&\{e^{\frac{1}{24}E_{2}(\tau)\cdot\frac{1}{30}c_{2}(W_{i})} \widehat{A}(TM,\nabla^{TM}){\rm ch}([\triangle(M)\otimes \Theta_1(T_{C}M)\\
&+2^{2k-1}\Theta_2(T_{C}M)+2^{2k-1}\Theta_3(T_{C}M)])\varphi(\tau)^{(8)}ch(\mathcal{V}_{i})\}^{(4p)}.\nonumber
\end{align}
Let $\{\pm 2\pi\sqrt{-1}x_j\}~(1\leq j \leq 2k-1)$ be the formal Chern roots for $T_\mathbf{C}M$, then we have
\begin{align}
&{\widehat{A}(TM,\nabla^{TM})}{\rm ch}(\triangle(M)){\rm ch}(\Theta_1(T_\mathbf{C}M))=\prod_{j=1}^{2k-1}\frac{2 x_j\theta'(0,\tau)}{\theta(x_j,\tau)}\cdot\frac{\theta_1(x_j,\tau)}{\theta_1(0,\tau)},\\
&{\widehat{A}(TM,\nabla^{TM})}{\rm ch}(2^{2k-1}\Theta_2(T_\mathbf{C}M))=\prod_{j=1}^{2k-1}\frac{2 x_j\theta'(0,\tau)}{\theta(x_j,\tau)}\cdot\frac{\theta_2(x_j,\tau)}{\theta_2(0,\tau)},\\
&{\widehat{A}(TM,\nabla^{TM})}{\rm ch}(2^{2k-1}\Theta_3(T_\mathbf{C}M))=\prod_{j=1}^{2k-1}\frac{2 x_j\theta'(0,\tau)}{\theta(x_j,\tau)}\cdot\frac{\theta_3(x_j,\tau)}{\theta_3(0,\tau)}.
\end{align}
Hence
\begin{align}
Q(M,P_{i},\tau) = & \left\{ e^{\frac{1}{24}E_{2}(\tau)\cdot\frac{1}{30}c_{2}(W_{i})} \cdot
\prod_{j=1}^{2k-1}\frac{2 x_j\theta'(0,\tau)}{\theta(x_j,\tau)}
\Big(
\frac{\theta_1(x_j,\tau)}{\theta_1(0,\tau)}
+ \frac{\theta_2(x_j,\tau)}{\theta_2(0,\tau)}
+\frac{\theta_3(x_j,\tau)}{\theta_3(0,\tau)}
\Big)\right.  \\
& \left. \cdot \frac{1}{2}\Big(\prod_{l=1}^{8}\theta_{1}(y_{l}^{i}, \tau)+\prod_{l=1}^{8}\theta_{2}(y_{l}^{i}, \tau)+\prod_{l=1}^{8}\theta_{3}(y_{l}^{i}, \tau)\Big) \vphantom{\prod_{j=1}^{2k-1}} \right\}^{(4p)}.\nonumber
\end{align}
From (2.16)-(2.28), we deduce that
\begin{align}
&Q(M, P_{i},\tau+1) =Q(M, \tau),\\
&Q\Big(M, P_{i},-\frac{1}{\tau}\Big) =\tau^{2p+4}Q(M, \tau).
\end{align}
According to the above Lemma, we have
\begin{align}
&Q(\nabla^{TM},P_{i},g,d,\tau+1)=Q(\nabla^{TM},P_{i},g,d,\tau),\\
&Q\Big(\nabla^{TM},P_{i},g,d,-\frac{1}{\tau}\Big) =\tau^{2r+2p+4}Q(\nabla^{TM},P_{i},g,d,\tau)=\tau^{2k+4}Q(\nabla^{TM},P_{i},g,d,\tau),
\end{align}
where $c_3(E_C,g,d)=0.$
So $Q(\nabla^{TM},P_{i},g,d,\tau)$ is a modular form over $\mathrm{SL}(2,\mathbf{Z})$ with the weight $2r+2p+4=2k+4$.
\end{proof}

\begin{thm}
For $7$-dimensional spin manifold, when $c_3(E,g,d)=0,$ we conclude that
  \begin{equation}
   \begin{aligned}
    &\{ e^{\frac{1}{720}c_{2}(W_{i})}{\widehat{A}(TM,\nabla^{TM})}[{\rm ch}(\triangle(M)\otimes 2\widetilde{T_{\mathbf{C}}M}){\rm ch}(\triangle(E),g^{\triangle(E)},d)+{\rm ch}(\triangle(M))\\
    &\cdot{\rm ch}(\Delta(E)\otimes(2\wedge^2\widetilde{E_\mathbf{C}}-\widetilde{E_\mathbf{C}}
    \otimes\widetilde{E_\mathbf{C}}+\widetilde{E_\mathbf{C}}),g,d)]+8e^{\frac{1}{720}c_{2}(W_{i})}\widehat{A}(TM,\nabla^{TM})\\
    &\cdot[{\rm ch}(\widetilde{T_{\mathbf{C}}M}+\wedge^2\widetilde{T_{\mathbf{C}}M}){\rm ch}(\triangle(E),g^{\triangle(E)},d)+{\rm ch}(\Delta(E)\otimes(2\wedge^2\widetilde{E_\mathbf{C}}-\widetilde{E_\mathbf{C}}\\
    & \otimes\widetilde{E_\mathbf{C}}+\widetilde{E_\mathbf{C}}),g,d)]+e^{\frac{1}{720}c_{2}(W_{i})}[{\rm ch}(-8+W_i)-\tfrac{1}{30}c_{2}(W_{i})][{\widehat{A}(TM,\nabla^{TM})}\\
    &\cdot{\rm ch}(\triangle(M)){\rm ch}(\triangle(E),g^{\triangle(E)},d)+8{\widehat{A}(TM,\nabla^{TM})}{\rm ch}(\triangle(E),g^{\triangle(E)},d)]\}^{(7)}\\
  &=480 \{ e^{\frac{1}{720}c_{2}(W_{i})}[{\widehat{A}(TM,\nabla^{TM})}{\rm ch}(\triangle(M)){\rm ch}(\triangle(E),g^{\triangle(E)},d) +8{\widehat{A}(TM,\nabla^{TM})}\\
    &\cdot{\rm ch}(\triangle(E),g^{\triangle(E)},d)]\}^{(7)},
   \end{aligned}
  \end{equation}
  and
  \begin{equation}
    \begin{aligned}
&\{e^{\frac{1}{720}c_{2}(W_{i})}\widehat{A}(TM,\nabla^{TM})[{\rm ch}(\triangle(M)\otimes(2\widetilde{T_\mathbf{C}M}+\wedge^2\widetilde{T_\mathbf{C}M}
+\widetilde{T_\mathbf{C}M}\otimes\widetilde{T_\mathbf{C}M}\\
&+S^2\widetilde{T_\mathbf{C}M})){\rm ch}(\triangle(E),g^{\triangle(E)},d)+{\rm ch}(\triangle(M)\otimes2\widetilde{T_{\mathbf{C}}M}){\rm ch}(\triangle(E)\otimes(2\wedge^2\widetilde{E_\mathbf{C}}\\
&-\widetilde{E_\mathbf{C}}\otimes\widetilde{E_\mathbf{C}}+\widetilde{E_\mathbf{C}}),g,d)+{\rm ch}(\triangle(M)){\rm ch}(\triangle(E)\otimes(\wedge^2\widetilde{E_\mathbf{C}}\otimes\wedge^2\widetilde{E_\mathbf{C}}
+2\wedge^4\widetilde{E_\mathbf{C}}\\
&-2\widetilde{E_\mathbf{C}}\otimes\wedge^3\widetilde{E_\mathbf{C}}
+2\widetilde{E_\mathbf{C}}\otimes\wedge^2\widetilde{E_\mathbf{C}}-\widetilde{E_\mathbf{C}}\otimes\widetilde{E_\mathbf{C}}\otimes\widetilde{E_\mathbf{C}}
+\widetilde{E_\mathbf{C}}+\wedge^2\widetilde{E_\mathbf{C}}),g,d)]\\
&+8e^{\frac{1}{720}c_{2}(W_{i})}\widehat{A}(TM,\nabla^{TM})[{\rm ch}(\wedge^4\widetilde{T_\mathbf{C}M}+\wedge^2\widetilde{T_\mathbf{C}M}\otimes\widetilde{T_\mathbf{C}M}+\widetilde{T_\mathbf{C}M}\otimes\widetilde{T_\mathbf{C}M}\\
&+S^2\widetilde{T_\mathbf{C}M}+\widetilde{T_\mathbf{C}M}){\rm ch}(\triangle(E),g^{\triangle(E)},d)
+{\rm ch}(\widetilde{T_{\mathbf{C}}M}+\wedge^2\widetilde{T_{\mathbf{C}}M}){\rm ch}(\triangle(E)\\
&\otimes(2\wedge^2\widetilde{E_\mathbf{C}}-\widetilde
{E_\mathbf{C}}\otimes\widetilde{E_\mathbf{C}}+\widetilde{E_\mathbf{C}}),g,d)
+{\rm ch}(\triangle(E)\otimes(\wedge^2\widetilde{E_\mathbf{C}}\otimes\wedge^2
\widetilde{E_\mathbf{C}}
+2\wedge^4\widetilde{E_\mathbf{C}}\\
&-2\widetilde{E_\mathbf{C}}\otimes\wedge^3
\widetilde{E_\mathbf{C}}
+2\widetilde{E_\mathbf{C}}\otimes\wedge^2\widetilde{E_\mathbf{C}}
-\widetilde{E_\mathbf{C}}\otimes\widetilde{E_\mathbf{C}}\otimes\widetilde{E_\mathbf{C}}
+\widetilde{E_\mathbf{C}}+\wedge^2\widetilde{E_\mathbf{C}}),g,d)]\\
&+e^{\frac{1}{720}c_{2}(W_{i})}[{\rm ch}(-8+W_i)-\tfrac{1}{30}c_{2}(W_{i})][\widehat{A}(TM,\nabla^{TM}){\rm ch}(\triangle(M)\otimes 2\widetilde{T_{\mathbf{C}}M})\\
&\cdot{\rm ch}(\triangle(E),g^{\triangle(E)},d)+\widehat{A}(TM,\nabla^{TM}){\rm ch}(\triangle(M)){\rm ch}(\Delta(E)\otimes(2\wedge^2\widetilde{E_\mathbf{C}}-\widetilde{E_\mathbf{C}}\\
    &\otimes\widetilde{E_\mathbf{C}}+\widetilde{E_\mathbf{C}}),g,d)+8\widehat{A}(TM,\nabla^{TM}){\rm ch}(\widetilde{T_{\mathbf{C}}M}+\wedge^2\widetilde{T_{\mathbf{C}}M}){\rm ch}(\triangle(E),g^{\triangle(E)},d)\\
    &+8\widehat{A}(TM,\nabla^{TM}){\rm ch}(\Delta(E)\otimes(2\wedge^2\widetilde{E_\mathbf{C}}-\widetilde{E_\mathbf{C}}
    \otimes\widetilde{E_\mathbf{C}}+\widetilde{E_\mathbf{C}}),g,d)]+e^{\frac{1}{720}c_{2}(W_{i})}\\
    &\cdot[{\rm ch}(20-8W_i+\overline{W_{i}})-\tfrac{1}{30}c_{2}(W_{i}){\rm ch}(-8+W_i)+\tfrac{1}{2}(-6+\tfrac{1}{30}c_{2}(W_{i}))\cdot\tfrac{1}{30}c_{2}(W_{i})]\\
    &\cdot[{\widehat{A}(TM,\nabla^{TM})}{\rm ch}(\triangle(M)){\rm ch}(\triangle(E),g^{\triangle(E)},d) +8{\widehat{A}(TM,\nabla^{TM})}\\
    &\cdot{\rm ch}(\triangle(E),g^{\triangle(E)},d)]\}^{(7)}=61920 \{ e^{\frac{1}{720}c_{2}(W_{i})}[{\widehat{A}(TM,\nabla^{TM})}{\rm ch}(\triangle(M))\\
&\cdot{\rm ch}(\triangle(E),g^{\triangle(E)},d) +8{\widehat{A}(TM,\nabla^{TM})}{\rm ch}(\triangle(E),g^{\triangle(E)},d)]\}^{(7)}.
    \end{aligned}
  \end{equation}
\end{thm}

\begin{proof}
A direct computation shows that
\begin{equation}
\begin{aligned}
e^{\frac{1}{24}E_{2}(\tau)\cdot\frac{1}{30}c_{2}(W_{i})}=&e^{\frac{1}{720}c_{2}(W_{i})}-qe^{\frac{1}{720}c_{2}(W_{i})}\cdot\tfrac{1}{30}c_{2}(W_{i})\\
&+q^2e^{\frac{1}{720}c_{2}(W_{i})}\cdot\tfrac{1}{2}(-6+\tfrac{1}{30}c_{2}(W_{i}))\cdot\tfrac{1}{30}c_{2}(W_{i})+\cdot\cdot\cdot,
\end{aligned}
\end{equation}
\begin{equation}
\begin{aligned}
&\widehat{A}(TM,\nabla^{TM}){\rm ch}([\triangle(M)\otimes \Theta_1(T_{C}M)+2^{2k-1}\Theta_2(T_{C}M)+2^{2k-1}\Theta_3(T_{C}M)])\\
=&{\widehat{A}(TM,\nabla^{TM})}{\rm ch}(\triangle(M)){\rm ch}[1+q\cdot2\widetilde{T_{\mathbf{C}}M}+q^2(2\widetilde{T_{\mathbf{C}}M}+\wedge^2\widetilde{T_\mathbf{C}M}+\widetilde{T_\mathbf{C}M}\\
&\otimes\widetilde{T_\mathbf{C}M})+S^2\widetilde{T_\mathbf{C}M}]+2^{2k-1}{\widehat{A}(TM,\nabla^{TM})}{\rm ch}[1+q(\widetilde{T_{\mathbf{C}}M}+\wedge^2\widetilde{T_{\mathbf{C}}M})\\
&+q^2(\wedge^4\widetilde{T_\mathbf{C}M}+\wedge^2\widetilde{T_\mathbf{C}M}\otimes\widetilde{T_\mathbf{C}M}+\widetilde{T_\mathbf{C}M}\otimes\widetilde{T_\mathbf{C}M}+S^2\widetilde{T_\mathbf{C}M}+\widetilde{T_\mathbf{C}M})]+\cdot\cdot\cdot,
\end{aligned}
\end{equation}
\begin{equation}
\begin{aligned}
{\rm ch}(Q(E),g^{Q(E)},d,\tau)=&{\rm ch}(\triangle(E),g^{\triangle(E)},d)+q{\rm ch}(\triangle(E)\otimes(2\wedge^2\widetilde{E_\mathbf{C}}-\widetilde{E_\mathbf{C}}\otimes\widetilde{E_\mathbf{C}}\\
&+\widetilde{E_\mathbf{C}}),g,d)+q^2{\rm ch}(\triangle(E)\otimes(\wedge^2\widetilde{E_\mathbf{C}}\otimes\wedge^2\widetilde{E_\mathbf{C}}+2\wedge^4\widetilde
{E_\mathbf{C}}\\
&-2\widetilde{E_\mathbf{C}}\otimes\wedge^3\widetilde{E_\mathbf{C}}+2\widetilde
{E_\mathbf{C}}\otimes\wedge^2\widetilde{E_\mathbf{C}}
-\widetilde{E_\mathbf{C}}\otimes\widetilde
{E_\mathbf{C}}\otimes\widetilde{E_\mathbf{C}}\\
&+\widetilde{E_\mathbf{C}}+\wedge^2\widetilde{E_\mathbf{C}}),g,d)+\cdot\cdot\cdot,
\end{aligned}
\end{equation}
and
\begin{equation}
\begin{aligned}
\varphi(\tau)^{(8)}ch(\mathcal{V}_{i})=1+q{\rm ch}(-8+W_i)+q^2{\rm ch}(20-8W_i+\overline{W_{i}})+\cdot\cdot\cdot.
\end{aligned}
\end{equation}
We thus get
\begin{equation}
\begin{aligned}
&Q(\nabla^{TM},P_{i},g,d,\tau)\\
=&\{ e^{\frac{1}{720}c_{2}(W_{i})}[{\widehat{A}(TM,\nabla^{TM})}{\rm ch}(\triangle(M)){\rm ch}(\triangle(E),g^{\triangle(E)},d) +2^{2k-1}{\widehat{A}(TM,\nabla^{TM})}\\
    &\cdot{\rm ch}(\triangle(E),g^{\triangle(E)},d)]\}^{(4k-1)}+q \{ e^{\frac{1}{720}c_{2}(W_{i})}{\widehat{A}(TM,\nabla^{TM})}[{\rm ch}(\triangle(M)\otimes2\widetilde{T_{\mathbf{C}}M})\\
    &\cdot{\rm ch}(\triangle(E),g^{\triangle(E)},d)+{\rm ch}(\triangle(M)){\rm ch}(\Delta(E)\otimes(2\wedge^2\widetilde{E_\mathbf{C}}-\widetilde{E_\mathbf{C}}
    \otimes\widetilde{E_\mathbf{C}}+\widetilde{E_\mathbf{C}})\\
    &,g,d)]+2^{2k-1}e^{\frac{1}{720}c_{2}(W_{i})}\widehat{A}(TM,\nabla^{TM})[{\rm ch}(\widetilde{T_{\mathbf{C}}M}+\wedge^2\widetilde{T_{\mathbf{C}}M}){\rm ch}(\triangle(E),g^{\triangle(E)},d)\\
    &+{\rm ch}(\Delta(E)\otimes(2\wedge^2\widetilde{E_\mathbf{C}}-\widetilde{E_\mathbf{C}}\otimes\widetilde{E_\mathbf{C}}+\widetilde{E_\mathbf{C}}),g,d)]+e^{\frac{1}{720}c_{2}(W_{i})}[{\rm ch}(-8+W_i)\\
    &-\tfrac{1}{30}c_{2}(W_{i})][{\widehat{A}(TM,\nabla^{TM})}{\rm ch}(\triangle(M)){\rm ch}(\triangle(E),g^{\triangle(E)},d)+2^{2k-1}{\widehat{A}(TM,\nabla^{TM})}\\
    &\cdot{\rm ch}(\triangle(E),g^{\triangle(E)},d)]\}^{(4k-1)}+q^2\{e^{\frac{1}{720}c_{2}(W_{i})}\widehat{A}(TM,\nabla^{TM})[{\rm ch}(\triangle(M)\\
    &\otimes(2\widetilde{T_\mathbf{C}M}+\wedge^2\widetilde{T_\mathbf{C}M}
+\widetilde{T_\mathbf{C}M}\otimes\widetilde{T_\mathbf{C}M}+S^2\widetilde{T_\mathbf{C}M})){\rm ch}(\triangle(E),g^{\triangle(E)},d)\\
&+{\rm ch}(\triangle(M)\otimes2\widetilde{T_{\mathbf{C}}M}){\rm ch}(\triangle(E)\otimes(2\wedge^2\widetilde{E_\mathbf{C}}-\widetilde{E_\mathbf{C}}\otimes\widetilde{E_\mathbf{C}}+\widetilde{E_\mathbf{C}}),g,d)\\
&+{\rm ch}(\triangle(M)){\rm ch}(\triangle(E)\otimes(\wedge^2\widetilde{E_\mathbf{C}}\otimes\wedge^2\widetilde{E_\mathbf{C}}
+2\wedge^4\widetilde{E_\mathbf{C}}-2\widetilde{E_\mathbf{C}}\otimes\wedge^3\widetilde{E_\mathbf{C}}\\
&+2\widetilde{E_\mathbf{C}}\otimes\wedge^2\widetilde{E_\mathbf{C}}-\widetilde{E_\mathbf{C}}\otimes\widetilde{E_\mathbf{C}}\otimes\widetilde{E_\mathbf{C}}
+\widetilde{E_\mathbf{C}}+\wedge^2\widetilde{E_\mathbf{C}}),g,d)]+2^{2k-1}e^{\frac{1}{720}c_{2}(W_{i})}\\
&\cdot \widehat{A}(TM,\nabla^{TM})[{\rm ch}(\wedge^4\widetilde{T_\mathbf{C}M}+\wedge^2\widetilde{T_\mathbf{C}M}\otimes\widetilde{T_\mathbf{C}M}+\widetilde{T_\mathbf{C}M}\otimes\widetilde{T_\mathbf{C}M}+S^2\widetilde{T_\mathbf{C}M}\\
&+\widetilde{T_\mathbf{C}M}){\rm ch}(\triangle(E),g^{\triangle(E)},d)
+{\rm ch}(\widetilde{T_{\mathbf{C}}M}+\wedge^2\widetilde{T_{\mathbf{C}}M}){\rm ch}(\triangle(E)\otimes(2\wedge^2\widetilde{E_\mathbf{C}}\\
&-\widetilde{E_\mathbf{C}}\otimes\widetilde{E_\mathbf{C}}+\widetilde{E_\mathbf{C}}),g,d)
+{\rm ch}(\triangle(E)\otimes(\wedge^2\widetilde{E_\mathbf{C}}\otimes\wedge^2\widetilde{E_\mathbf{C}}
+2\wedge^4\widetilde{E_\mathbf{C}}-2\widetilde{E_\mathbf{C}}   \end{aligned}
\end{equation}
\begin{equation}
\begin{aligned}
&\otimes\wedge^3\widetilde{E_\mathbf{C}}
+2\widetilde{E_\mathbf{C}}\otimes\wedge^2\widetilde{E_\mathbf{C}}
-\widetilde{E_\mathbf{C}}\otimes\widetilde{E_\mathbf{C}}\otimes\widetilde{E_\mathbf{C}}
+\widetilde{E_\mathbf{C}}+\wedge^2\widetilde{E_\mathbf{C}}),g,d)]\\
&+e^{\frac{1}{720}c_{2}(W_{i})}[{\rm ch}(-8+W_i)-\tfrac{1}{30}c_{2}(W_{i})][\widehat{A}(TM,\nabla^{TM}){\rm ch}(\triangle(M)\otimes 2\widetilde{T_{\mathbf{C}}M})\\
&\cdot{\rm ch}(\triangle(E),g^{\triangle(E)},d)+\widehat{A}(TM,\nabla^{TM}){\rm ch}(\triangle(M)){\rm ch}(\Delta(E)\otimes(2\wedge^2\widetilde{E_\mathbf{C}}-\widetilde{E_\mathbf{C}}\\
    &\otimes\widetilde{E_\mathbf{C}}+\widetilde{E_\mathbf{C}}),g,d)+2^{2k-1}\widehat{A}(TM,\nabla^{TM}){\rm ch}(\widetilde{T_{\mathbf{C}}M}+\wedge^2\widetilde{T_{\mathbf{C}}M}){\rm ch}(\triangle(E),g^{\triangle(E)},d)\\
    &+2^{2k-1}\widehat{A}(TM,\nabla^{TM}){\rm ch}(\Delta(E)\otimes(2\wedge^2\widetilde{E_\mathbf{C}}-\widetilde{E_\mathbf{C}}
    \otimes\widetilde{E_\mathbf{C}}+\widetilde{E_\mathbf{C}}),g,d)]+e^{\frac{1}{720}c_{2}(W_{i})}\\
    &\cdot[{\rm ch}(20-8W_i+\overline{W_{i}})-\tfrac{1}{30}c_{2}(W_{i}){\rm ch}(-8+W_i)+\tfrac{1}{2}(-6+\tfrac{1}{30}c_{2}(W_{i}))\cdot\tfrac{1}{30}c_{2}(W_{i})]\\
    &\cdot[{\widehat{A}(TM,\nabla^{TM})}{\rm ch}(\triangle(M)){\rm ch}(\triangle(E),g^{\triangle(E)},d) +2^{2k-1}{\widehat{A}(TM,\nabla^{TM})}\\
    &\cdot{\rm ch}(\triangle(E),g^{\triangle(E)},d)]\}^{(4k-1)}.\nonumber
\end{aligned}
\end{equation}

It is well known that modular forms over $\mathrm{SL}(2,\mathbf{Z})$ can be expressed as polynomials of the Einsentein series $E_{4}(\tau)$ and $E_{6}(\tau),$ where
\begin{align}
&E_{4}(\tau)=1+240q+2160q^{2}+6720q^{3}+\cdot\cdot\cdot,\\
&E_{6}(\tau)=1-504q-16632q^{2}-122976q^{3}+\cdot\cdot\cdot.
\end{align}
Their weights are $4$ and $6$ respectively.

When ${\rm dim}M=7$, then $Q(\nabla^{TM},P_{i},g,d,\tau)$ is a modular form over $SL(2,{\bf Z})$ with the weight $8$, it must be a multiple of
\begin{align}
E_{4}(\tau)^{2}=1+480q+61920q^{2}+\cdot\cdot\cdot.
\end{align}

Theorem 3.3 follows from a comparison of the coefficients of $q$ and $q^2$ in (3.37) and (3.40).
\end{proof}

\begin{thm}
For $11$-dimensional spin manifold, when $c_3(E,g,d)=0,$ we conclude that
  \begin{equation}
   \begin{aligned}
    &\{ e^{\frac{1}{720}c_{2}(W_{i})}{\widehat{A}(TM,\nabla^{TM})}[{\rm ch}(\triangle(M)\otimes 2\widetilde{T_{\mathbf{C}}M}){\rm ch}(\triangle(E),g^{\triangle(E)},d)+{\rm ch}(\triangle(M))\\
    &\cdot{\rm ch}(\Delta(E)\otimes(2\wedge^2\widetilde{E_\mathbf{C}}-\widetilde{E_\mathbf{C}}
    \otimes\widetilde{E_\mathbf{C}}+\widetilde{E_\mathbf{C}}),g,d)]+32e^{\frac{1}{720}c_{2}(W_{i})}\widehat{A}(TM,\nabla^{TM})\\
    &\cdot[{\rm ch}(\widetilde{T_{\mathbf{C}}M}+\wedge^2\widetilde{T_{\mathbf{C}}M}){\rm ch}(\triangle(E),g^{\triangle(E)},d)+{\rm ch}(\Delta(E)\otimes(2\wedge^2\widetilde{E_\mathbf{C}}-\widetilde{E_\mathbf{C}}\\
    & \otimes\widetilde{E_\mathbf{C}}+\widetilde{E_\mathbf{C}}),g,d)]+e^{\frac{1}{720}c_{2}(W_{i})}[{\rm ch}(-8+W_i)-\tfrac{1}{30}c_{2}(W_{i})][{\widehat{A}(TM,\nabla^{TM})}\\
    &\cdot{\rm ch}(\triangle(M)){\rm ch}(\triangle(E),g^{\triangle(E)},d)+32{\widehat{A}(TM,\nabla^{TM})}{\rm ch}(\triangle(E),g^{\triangle(E)},d)]\}^{(11)}\\
  &=-264\{ e^{\frac{1}{720}c_{2}(W_{i})}[{\widehat{A}(TM,\nabla^{TM})}{\rm ch}(\triangle(M)){\rm ch}(\triangle(E),g^{\triangle(E)},d) +32{\widehat{A}(TM,\nabla^{TM})}\\
    &\cdot{\rm ch}(\triangle(E),g^{\triangle(E)},d)]\}^{(11)},
   \end{aligned}
  \end{equation}
  and
  \begin{equation}
    \begin{aligned}
&\{e^{\frac{1}{720}c_{2}(W_{i})}\widehat{A}(TM,\nabla^{TM})[{\rm ch}(\triangle(M)\otimes(2\widetilde{T_\mathbf{C}M}+\wedge^2\widetilde{T_\mathbf{C}M}
+\widetilde{T_\mathbf{C}M}\otimes\widetilde{T_\mathbf{C}M}\\
&+S^2\widetilde{T_\mathbf{C}M})){\rm ch}(\triangle(E),g^{\triangle(E)},d)+{\rm ch}(\triangle(M)\otimes2\widetilde{T_{\mathbf{C}}M}){\rm ch}(\triangle(E)\otimes(2\wedge^2\widetilde{E_\mathbf{C}}\\
&-\widetilde{E_\mathbf{C}}\otimes\widetilde{E_\mathbf{C}}+\widetilde{E_\mathbf{C}}),g,d)+{\rm ch}(\triangle(M)){\rm ch}(\triangle(E)\otimes(\wedge^2\widetilde{E_\mathbf{C}}\otimes\wedge^2\widetilde{E_\mathbf{C}}
+2\wedge^4\widetilde{E_\mathbf{C}}\\
&-2\widetilde{E_\mathbf{C}}\otimes\wedge^3\widetilde{E_\mathbf{C}}
+2\widetilde{E_\mathbf{C}}\otimes\wedge^2\widetilde{E_\mathbf{C}}-\widetilde{E_\mathbf{C}}\otimes\widetilde{E_\mathbf{C}}\otimes\widetilde{E_\mathbf{C}}
+\widetilde{E_\mathbf{C}}+\wedge^2\widetilde{E_\mathbf{C}}),g,d)]\\
&+32e^{\frac{1}{720}c_{2}(W_{i})}\widehat{A}(TM,\nabla^{TM})[{\rm ch}(\wedge^4\widetilde{T_\mathbf{C}M}+\wedge^2\widetilde{T_\mathbf{C}M}\otimes\widetilde{T_\mathbf{C}M}+\widetilde{T_\mathbf{C}M}\otimes\widetilde{T_\mathbf{C}M}\\
&+S^2\widetilde{T_\mathbf{C}M}+\widetilde{T_\mathbf{C}M}){\rm ch}(\triangle(E),g^{\triangle(E)},d)
+{\rm ch}(\widetilde{T_{\mathbf{C}}M}+\wedge^2\widetilde{T_{\mathbf{C}}M}){\rm ch}(\triangle(E)\\
&\otimes(2\wedge^2\widetilde{E_\mathbf{C}}-\widetilde
{E_\mathbf{C}}\otimes\widetilde{E_\mathbf{C}}+\widetilde{E_\mathbf{C}}),g,d)
+{\rm ch}(\triangle(E)\otimes(\wedge^2\widetilde{E_\mathbf{C}}\otimes\wedge^2
\widetilde{E_\mathbf{C}}
+2\wedge^4\widetilde{E_\mathbf{C}}\\
&-2\widetilde{E_\mathbf{C}}\otimes\wedge^3
\widetilde{E_\mathbf{C}}
+2\widetilde{E_\mathbf{C}}\otimes\wedge^2\widetilde{E_\mathbf{C}}
-\widetilde{E_\mathbf{C}}\otimes\widetilde{E_\mathbf{C}}\otimes\widetilde{E_\mathbf{C}}
+\widetilde{E_\mathbf{C}}+\wedge^2\widetilde{E_\mathbf{C}}),g,d)]\\
&+e^{\frac{1}{720}c_{2}(W_{i})}[{\rm ch}(-8+W_i)-\tfrac{1}{30}c_{2}(W_{i})][\widehat{A}(TM,\nabla^{TM}){\rm ch}(\triangle(M)\otimes 2\widetilde{T_{\mathbf{C}}M})\\
&\cdot{\rm ch}(\triangle(E),g^{\triangle(E)},d)+\widehat{A}(TM,\nabla^{TM}){\rm ch}(\triangle(M)){\rm ch}(\Delta(E)\otimes(2\wedge^2\widetilde{E_\mathbf{C}}-\widetilde{E_\mathbf{C}}\\
    &\otimes\widetilde{E_\mathbf{C}}+\widetilde{E_\mathbf{C}}),g,d)+32\widehat{A}(TM,\nabla^{TM}){\rm ch}(\widetilde{T_{\mathbf{C}}M}+\wedge^2\widetilde{T_{\mathbf{C}}M}){\rm ch}(\triangle(E),g^{\triangle(E)},d)\\
    &+32\widehat{A}(TM,\nabla^{TM}){\rm ch}(\Delta(E)\otimes(2\wedge^2\widetilde{E_\mathbf{C}}-\widetilde{E_\mathbf{C}}
    \otimes\widetilde{E_\mathbf{C}}+\widetilde{E_\mathbf{C}}),g,d)]+e^{\frac{1}{720}c_{2}(W_{i})}\\
    &\cdot[{\rm ch}(20-8W_i+\overline{W_{i}})-\tfrac{1}{30}c_{2}(W_{i}){\rm ch}(-8+W_i)+\tfrac{1}{2}(-6+\tfrac{1}{30}c_{2}(W_{i}))\cdot\tfrac{1}{30}c_{2}(W_{i})]\\
    &\cdot[{\widehat{A}(TM,\nabla^{TM})}{\rm ch}(\triangle(M)){\rm ch}(\triangle(E),g^{\triangle(E)},d) +32{\widehat{A}(TM,\nabla^{TM})}\\
    &\cdot{\rm ch}(\triangle(E),g^{\triangle(E)},d)]\}^{(11)}=-135432 \{ e^{\frac{1}{720}c_{2}(W_{i})}[{\widehat{A}(TM,\nabla^{TM})}{\rm ch}(\triangle(M))\\
&\cdot{\rm ch}(\triangle(E),g^{\triangle(E)},d) +32{\widehat{A}(TM,\nabla^{TM})}{\rm ch}(\triangle(E),g^{\triangle(E)},d)]\}^{(11)}.
    \end{aligned}
  \end{equation}
\end{thm}

\begin{proof}
When $k=3$, $Q(\nabla^{TM}, P_{i}, g, d, \tau)$ is a modular form on $\mathrm{SL}(2,\mathbf{Z})$ of weight $10.$
Clearly,
\begin{align}
E_{4}(\tau)E_{6}(\tau)=1-264q-135432q^{2}+\cdot\cdot\cdot.
\end{align}

By comparing the constant term and the coefficients of $q,$~$q^{2}$  we can get Theorem 3.4.
\end{proof}

\begin{thm}
For $15$-dimensional spin manifold, when $c_3(E,g,d)=0,$ we conclude that
   \begin{equation}
    \begin{aligned}
&\{e^{\frac{1}{720}c_{2}(W_{i})}\widehat{A}(TM,\nabla^{TM})[{\rm ch}(\triangle(M)\otimes(2\widetilde{T_\mathbf{C}M}+\wedge^2\widetilde{T_\mathbf{C}M}
+\widetilde{T_\mathbf{C}M}\otimes\widetilde{T_\mathbf{C}M}\\
&+S^2\widetilde{T_\mathbf{C}M})){\rm ch}(\triangle(E),g^{\triangle(E)},d)+{\rm ch}(\triangle(M)\otimes2\widetilde{T_{\mathbf{C}}M}){\rm ch}(\triangle(E)\otimes(2\wedge^2\widetilde{E_\mathbf{C}}\\
&-\widetilde{E_\mathbf{C}}\otimes\widetilde{E_\mathbf{C}}+\widetilde{E_\mathbf{C}}),g,d)+{\rm ch}(\triangle(M)){\rm ch}(\triangle(E)\otimes(\wedge^2\widetilde{E_\mathbf{C}}\otimes\wedge^2\widetilde{E_\mathbf{C}}
+2\wedge^4\widetilde{E_\mathbf{C}}\\
&-2\widetilde{E_\mathbf{C}}\otimes\wedge^3\widetilde{E_\mathbf{C}}
+2\widetilde{E_\mathbf{C}}\otimes\wedge^2\widetilde{E_\mathbf{C}}-\widetilde{E_\mathbf{C}}\otimes\widetilde{E_\mathbf{C}}\otimes\widetilde{E_\mathbf{C}}
+\widetilde{E_\mathbf{C}}+\wedge^2\widetilde{E_\mathbf{C}}),g,d)]   \end{aligned}
\end{equation}
\begin{equation}
\begin{aligned}
&+128e^{\frac{1}{720}c_{2}(W_{i})}\widehat{A}(TM,\nabla^{TM})[{\rm ch}(\wedge^4\widetilde{T_\mathbf{C}M}+\wedge^2\widetilde{T_\mathbf{C}M}\otimes\widetilde{T_\mathbf{C}M}+\widetilde{T_\mathbf{C}M}\otimes\widetilde{T_\mathbf{C}M}\\
&+S^2\widetilde{T_\mathbf{C}M}+\widetilde{T_\mathbf{C}M}){\rm ch}(\triangle(E),g^{\triangle(E)},d)
+{\rm ch}(\widetilde{T_{\mathbf{C}}M}+\wedge^2\widetilde{T_{\mathbf{C}}M}){\rm ch}(\triangle(E)\\
&\otimes(2\wedge^2\widetilde{E_\mathbf{C}}-\widetilde
{E_\mathbf{C}}\otimes\widetilde{E_\mathbf{C}}+\widetilde{E_\mathbf{C}}),g,d)
+{\rm ch}(\triangle(E)\otimes(\wedge^2\widetilde{E_\mathbf{C}}\otimes\wedge^2
\widetilde{E_\mathbf{C}}
+2\wedge^4\widetilde{E_\mathbf{C}}\\
&-2\widetilde{E_\mathbf{C}}\otimes\wedge^3
\widetilde{E_\mathbf{C}}
+2\widetilde{E_\mathbf{C}}\otimes\wedge^2\widetilde{E_\mathbf{C}}
-\widetilde{E_\mathbf{C}}\otimes\widetilde{E_\mathbf{C}}\otimes\widetilde{E_\mathbf{C}}
+\widetilde{E_\mathbf{C}}+\wedge^2\widetilde{E_\mathbf{C}}),g,d)]\\
&+e^{\frac{1}{720}c_{2}(W_{i})}[{\rm ch}(-8+W_i)-\tfrac{1}{30}c_{2}(W_{i})][\widehat{A}(TM,\nabla^{TM}){\rm ch}(\triangle(M)\otimes 2\widetilde{T_{\mathbf{C}}M})\\
&\cdot{\rm ch}(\triangle(E),g^{\triangle(E)},d)+\widehat{A}(TM,\nabla^{TM}){\rm ch}(\triangle(M)){\rm ch}(\Delta(E)\otimes(2\wedge^2\widetilde{E_\mathbf{C}}-\widetilde{E_\mathbf{C}}\\
    &\otimes\widetilde{E_\mathbf{C}}+\widetilde{E_\mathbf{C}}),g,d)+128\widehat{A}(TM,\nabla^{TM}){\rm ch}(\widetilde{T_{\mathbf{C}}M}+\wedge^2\widetilde{T_{\mathbf{C}}M}){\rm ch}(\triangle(E),g^{\triangle(E)},d)\\
    &+128\widehat{A}(TM,\nabla^{TM}){\rm ch}(\Delta(E)\otimes(2\wedge^2\widetilde{E_\mathbf{C}}-\widetilde{E_\mathbf{C}}
    \otimes\widetilde{E_\mathbf{C}}+\widetilde{E_\mathbf{C}}),g,d)]+e^{\frac{1}{720}c_{2}(W_{i})}\\
    &\cdot[{\rm ch}(20-8W_i+\overline{W_{i}})-\tfrac{1}{30}c_{2}(W_{i}){\rm ch}(-8+W_i)+\tfrac{1}{2}(-6+\tfrac{1}{30}c_{2}(W_{i}))\cdot\tfrac{1}{30}c_{2}(W_{i})]\\
    &\cdot[{\widehat{A}(TM,\nabla^{TM})}{\rm ch}(\triangle(M)){\rm ch}(\triangle(E),g^{\triangle(E)},d) +128{\widehat{A}(TM,\nabla^{TM})}\\
    &\cdot{\rm ch}(\triangle(E),g^{\triangle(E)},d)]\}^{(15)}=196560 \{ e^{\frac{1}{720}c_{2}(W_{i})}[{\widehat{A}(TM,\nabla^{TM})}{\rm ch}(\triangle(M))\\
&\cdot{\rm ch}(\triangle(E),g^{\triangle(E)},d) +128{\widehat{A}(TM,\nabla^{TM})}{\rm ch}(\triangle(E),g^{\triangle(E)},d)]\}^{(15)}\\
 &-24\{ e^{\frac{1}{720}c_{2}(W_{i})}{\widehat{A}(TM,\nabla^{TM})}[{\rm ch}(\triangle(M)\otimes 2\widetilde{T_{\mathbf{C}}M}){\rm ch}(\triangle(E),g^{\triangle(E)},d)\\
    &+{\rm ch}(\triangle(M)){\rm ch}(\Delta(E)\otimes(2\wedge^2\widetilde{E_\mathbf{C}}-\widetilde{E_\mathbf{C}}
    \otimes\widetilde{E_\mathbf{C}}+\widetilde{E_\mathbf{C}}),g,d)]+128e^{\frac{1}{720}c_{2}(W_{i})}\\
    &\cdot\widehat{A}(TM,\nabla^{TM})[{\rm ch}(\widetilde{T_{\mathbf{C}}M}+\wedge^2\widetilde{T_{\mathbf{C}}M}){\rm ch}(\triangle(E),g^{\triangle(E)},d)+{\rm ch}(\Delta(E)\otimes(2\wedge^2\widetilde{E_\mathbf{C}}\\
    &-\widetilde{E_\mathbf{C}} \otimes\widetilde{E_\mathbf{C}}+\widetilde{E_\mathbf{C}}),g,d)]+e^{\frac{1}{720}c_{2}(W_{i})}[{\rm ch}(-8+W_i)-\tfrac{1}{30}c_{2}(W_{i})][{\widehat{A}(TM,\nabla^{TM})}\\
    &\cdot{\rm ch}(\triangle(M)){\rm ch}(\triangle(E),g^{\triangle(E)},d)+128{\widehat{A}(TM,\nabla^{TM})}{\rm ch}(\triangle(E),g^{\triangle(E)},d)]\}^{(15)}.\nonumber
    \end{aligned}
  \end{equation}
\end{thm}

\begin{proof}
$Q(\nabla^{TM},P_{i},g,d,\tau)$ is a modular form of weight $12$ over $\mathrm{SL}(2,\mathbf{Z}),$ so
\begin{align}
Q(\nabla^{TM},P_{i},g,d,\tau)&=\lambda_{1}E_{4}(\tau)^{3}+\lambda_{2}E_{6}(\tau)^{2}\\
&=1+q(720\lambda_{1}-1008\lambda_{2})+q^2(179280\lambda_{1}+220752\lambda_{2})+\cdot\cdot\cdot.\nonumber
\end{align}
where $\lambda_1$, $\lambda_2$ is degree $12$ forms.

By (3.37) and (3.45) , we get two equations about $\lambda_1$, $\lambda_2$
\begin{equation}
\begin{aligned}
\lambda_1
=&\tfrac{1008}{1728}\{ e^{\frac{1}{720}c_{2}(W_{i})}[{\widehat{A}(TM,\nabla^{TM})}{\rm ch}(\triangle(M)){\rm ch}(\triangle(E),g^{\triangle(E)},d) +128\\
&\cdot{\widehat{A}(TM,\nabla^{TM})}{\rm ch}(\triangle(E),g^{\triangle(E)},d)]\}^{(15)}+\tfrac{1}{1728}\{ e^{\frac{1}{720}c_{2}(W_{i})}{\widehat{A}(TM,\nabla^{TM})}\\
 &\cdot[{\rm ch}(\triangle(M)\otimes 2\widetilde{T_{\mathbf{C}}M}){\rm ch}(\triangle(E),g^{\triangle(E)},d)+{\rm ch}(\triangle(M)){\rm ch}(\Delta(E)\otimes(2\wedge^2\widetilde{E_\mathbf{C}}\\
    &-\widetilde{E_\mathbf{C}}
    \otimes\widetilde{E_\mathbf{C}}+\widetilde{E_\mathbf{C}}),g,d)]+128e^{\frac{1}{720}c_{2}(W_{i})}\cdot\widehat{A}(TM,\nabla^{TM})[{\rm ch}(\widetilde{T_{\mathbf{C}}M}+\wedge^2\widetilde{T_{\mathbf{C}}M})\\
    &\cdot{\rm ch}(\triangle(E),g^{\triangle(E)},d)+{\rm ch}(\Delta(E)\otimes(2\wedge^2\widetilde{E_\mathbf{C}}-\widetilde{E_\mathbf{C}} \otimes\widetilde{E_\mathbf{C}}+\widetilde{E_\mathbf{C}}),g,d)]\\
    &+e^{\frac{1}{720}c_{2}(W_{i})}[{\rm ch}(-8+W_i)-\tfrac{1}{30}c_{2}(W_{i})][{\widehat{A}(TM,\nabla^{TM})}{\rm ch}(\triangle(M))\\
    &\cdot{\rm ch}(\triangle(E),g^{\triangle(E)},d)+128{\widehat{A}(TM,\nabla^{TM})}{\rm ch}(\triangle(E),g^{\triangle(E)},d)]\}^{(15)},
    \end{aligned}
  \end{equation}
  and
  \begin{equation}
\begin{aligned}
\lambda_2
=&\tfrac{720}{1728}\{ e^{\frac{1}{720}c_{2}(W_{i})}[{\widehat{A}(TM,\nabla^{TM})}{\rm ch}(\triangle(M)){\rm ch}(\triangle(E),g^{\triangle(E)},d) +128\\
&\cdot{\widehat{A}(TM,\nabla^{TM})}{\rm ch}(\triangle(E),g^{\triangle(E)},d)]\}^{(15)}-\tfrac{1}{1728}\{ e^{\frac{1}{720}c_{2}(W_{i})}{\widehat{A}(TM,\nabla^{TM})}\\
 &\cdot[{\rm ch}(\triangle(M)\otimes 2\widetilde{T_{\mathbf{C}}M}){\rm ch}(\triangle(E),g^{\triangle(E)},d)+{\rm ch}(\triangle(M)){\rm ch}(\Delta(E)\otimes(2\wedge^2\widetilde{E_\mathbf{C}}\\
    &-\widetilde{E_\mathbf{C}}
    \otimes\widetilde{E_\mathbf{C}}+\widetilde{E_\mathbf{C}}),g,d)]+128e^{\frac{1}{720}c_{2}(W_{i})}\cdot\widehat{A}(TM,\nabla^{TM})[{\rm ch}(\widetilde{T_{\mathbf{C}}M}+\wedge^2\widetilde{T_{\mathbf{C}}M})\\
    &\cdot{\rm ch}(\triangle(E),g^{\triangle(E)},d)+{\rm ch}(\Delta(E)\otimes(2\wedge^2\widetilde{E_\mathbf{C}}-\widetilde{E_\mathbf{C}} \otimes\widetilde{E_\mathbf{C}}+\widetilde{E_\mathbf{C}}),g,d)]\\
    &+e^{\frac{1}{720}c_{2}(W_{i})}[{\rm ch}(-8+W_i)-\tfrac{1}{30}c_{2}(W_{i})][{\widehat{A}(TM,\nabla^{TM})}{\rm ch}(\triangle(M))\\
    &\cdot{\rm ch}(\triangle(E),g^{\triangle(E)},d)+128{\widehat{A}(TM,\nabla^{TM})}{\rm ch}(\triangle(E),g^{\triangle(E)},d)]\}^{(15)}.
    \end{aligned}
  \end{equation}

Comparing the coefficients yields two equations.
\end{proof}

\subsection{Modular forms and Witten genus in $(4k-1)$-dimensional spin$^c$ manifolds}\hfill\\

Let $M$ be a closed oriented ${\rm spin}^{c}$ manifold of dimension $4k-1$.
Set $L$ be the complex line bundle associated to the given ${\rm spin^{c}}$ structure on $M.$
Denote by $c=c_1(L)$ the first Chern class of $L,$~$y=-\frac{\sqrt{-1}}{2\pi}c.$
Also, we use $L_{\bf{R}}$ for the notation of $L,$ when it is viewed as an oriented real plane bundle.
Let $\Theta(T_{\mathbf{C}}M,L_{\bf{R}}\otimes\bf{C})$ be the virtual complex vector bundle over $M$ defined by
\begin{equation}
    \begin{aligned}
        \Theta(T_{\mathbf{C}}M,L_{\bf{R}}\otimes\mathbf{C})=&\bigotimes _{n=1}^{\infty}S_{q^n}(\widetilde{T_{\mathbf{C}}M})\otimes
\bigotimes _{m=1}^{\infty}\wedge_{q^m}(\widetilde{L_{\bf{R}}\otimes\mathbf{C}})\\
&\otimes
\bigotimes _{r=1}^{\infty}\wedge_{-q^{r-\frac{1}{2}}}(\widetilde{L_{\bf{R}}\otimes\mathbf{C}})\otimes
\bigotimes _{s=1}^{\infty}\wedge_{q^{s-\frac{1}{2}}}(\widetilde{L_{\bf{R}}\otimes\mathbf{C}}).
    \end{aligned}
\end{equation}

When $c = 0$, the bundle $\Theta(T_{\mathbf{C}}M) = \bigotimes_{n=1}^{\infty} S_{q^{n}}(\widetilde{T_{\mathbf{C}}M})$ is called the Witten bundle over $M$. Following \cite{HY}, the corresponding Witten form in odd dimensions is defined as
\begin{equation}
\begin{aligned}
\widetilde{Q}(\nabla^{TM},g,d,\tau)=\{\widehat{A}(TM,\nabla^{TM}){\rm ch}(\Theta(T_{\mathbf{C}}M)){\rm ch}(Q(E),g^{Q(E)},d,\tau)\}^{(4k-1)}
\end{aligned}
\end{equation}
We begin by generalizing the form $\widetilde{Q}(\nabla^{TM},g,d,\tau)$ to the setting involving a complex line bundle $L$ and $E_{8}$ bundles.
This is achieved by considering
\begin{equation}
\begin{aligned}
\widetilde{Q}(\nabla^{TM},\nabla^{L},P_{i},g,d,\tau)=&\{e^{\frac{1}{24}E_{2}(\tau)\cdot\frac{1}{30}c_{2}(W_{i})} \widehat{A}(TM,\nabla^{TM}){\rm ch}(\Theta(T_{\mathbf{C}}M,L_{\bf{R}}\otimes\mathbf{C}))\\
&\cdot{\rm exp}(\frac{c}{2}){\rm ch}(Q(E),g^{Q(E)},d,\tau)\varphi(\tau)^{(8)}ch(\mathcal{V}_{i})\}^{(4k-1)}.
\end{aligned}
\end{equation}

Let $p_1$ denote the first Pontryagin class.
We obtain the following theorem.
\begin{thm}
Let ${\rm dim}M=4k-1,$ for $k\leq 4.$  If $3p_1(L)-p_1(M)=0$ and $c_3(E,g,d)=0,$ then $\widetilde{Q}(\nabla^{TM},\nabla^{L},P_{i},g,d,\tau)$ is a modular form over $\mathrm{SL}(2,\mathbf{Z})$ with the weight $2k+4.$
\end{thm}

\begin{thm}
For $7$-dimensional spin manifold, when $3p_1(L)-p_1(M)=0$ and $c_3(E,g,d)=0,$ we conclude that
 \begin{equation}
\begin{aligned}
& \{ {\widehat{A}(TM,\nabla^{TM})}{\rm exp}(\frac{c}{2}){\rm ch}(\triangle(E),g^{\triangle(E)},d)[-e^{\frac{1}{720}c_{2}(W_{i})}\cdot\tfrac{1}{30}c_{2}(W_{i})+e^{\frac{1}{720}c_{2}(W_{i})}\\
&\cdot{\rm ch}(-8+W_i)+e^{\frac{1}{720}c_{2}(W_{i})}{\rm ch}(\widetilde{T_\mathbf{C}M}+\widetilde{E_\mathbf{C}}-\widetilde{E_\mathbf{C}}
    \otimes\widetilde{E_\mathbf{C}}+2\wedge^2\widetilde{E_\mathbf{C}}+\widetilde{L_{\bf{R}}\otimes\mathbf{C}}\\
    &-(\widetilde{L_{\bf{R}}\otimes\mathbf{C}})
    \otimes(\widetilde{L_{\bf{R}}\otimes\mathbf{C}})+2\wedge^2\widetilde{L_{\bf{R}}\otimes\mathbf{C}})]\}^{(7)}=480\{ e^{\frac{1}{720}c_{2}(W_{i})}{\widehat{A}(TM,\nabla^{TM})}\\
    &\cdot{\rm exp}(\frac{c}{2}){\rm ch}(\triangle(E),g^{\triangle(E)},d) \}^{(7)},
\end{aligned}
\end{equation}
  and
 \begin{equation}
\begin{aligned}
&\{ {\widehat{A}(TM,\nabla^{TM})}{\rm exp}(\frac{c}{2}){\rm ch}(\triangle(E),g^{\triangle(E)},d)[e^{\frac{1}{720}c_{2}(W_{i})}\cdot\tfrac{1}{2}(-6+\tfrac{1}{30}c_{2}(W_{i}))\cdot\tfrac{1}{30}c_{2}(W_{i})\\
&-e^{\frac{1}{720}c_{2}(W_{i})}\cdot\tfrac{1}{30}c_{2}(W_{i}){\rm ch}(-8+W_i)+e^{\frac{1}{720}c_{2}(W_{i})}{\rm ch}(20-8W_i+\overline{W_{i}})-e^{\frac{1}{720}c_{2}(W_{i})}\\
&\cdot\tfrac{1}{30}c_{2}(W_{i}){\rm ch}(\widetilde{T_\mathbf{C}M}+\widetilde{E_\mathbf{C}}-\widetilde{E_\mathbf{C}}
    \otimes\widetilde{E_\mathbf{C}}+2\wedge^2\widetilde{E_\mathbf{C}}+\widetilde{L_{\bf{R}}\otimes\mathbf{C}}-(\widetilde{L_{\bf{R}}\otimes\mathbf{C}})\\
    & \otimes(\widetilde{L_{\bf{R}}\otimes\mathbf{C}})+2\wedge^2\widetilde{L_{\bf{R}}\otimes\mathbf{C}})+e^{\frac{1}{720}c_{2}(W_{i})}{\rm ch}(-8+W_i){\rm ch}(\widetilde{T_\mathbf{C}M}+\widetilde{E_\mathbf{C}}\\
&-\widetilde{E_\mathbf{C}}
    \otimes\widetilde{E_\mathbf{C}}+2\wedge^2\widetilde{E_\mathbf{C}}+\widetilde{L_{\bf{R}}\otimes\mathbf{C}}-(\widetilde{L_{\bf{R}}\otimes\mathbf{C}})
    \otimes(\widetilde{L_{\bf{R}}\otimes\mathbf{C}})+2\wedge^2\widetilde{L_{\bf{R}}\otimes\mathbf{C}})\\
    &+e^{\frac{1}{720}c_{2}(W_{i})}{\rm ch}(\widetilde{T_\mathbf{C}M}+S^2\widetilde{T_\mathbf{C}M}+\widetilde{T_\mathbf{C}M}\otimes(\widetilde{E_\mathbf{C}}-\widetilde{E_\mathbf{C}}
    \otimes\widetilde{E_\mathbf{C}}+2\wedge^2\widetilde{E_\mathbf{C}}+\widetilde{L_{\bf{R}}\otimes\mathbf{C}}\\
    &-(\widetilde{L_{\bf{R}}\otimes\mathbf{C}})
    \otimes(\widetilde{L_{\bf{R}}\otimes\mathbf{C}})+2\wedge^2\widetilde{L_{\bf{R}}\otimes\mathbf{C}})+\wedge^2\widetilde{E_\mathbf{C}}\otimes\wedge^2\widetilde{E_\mathbf{C}}
+2\wedge^4\widetilde{E_\mathbf{C}}-2\widetilde{E_\mathbf{C}}\\
&\otimes\wedge^3\widetilde{E_\mathbf{C}}
+2\widetilde{E_\mathbf{C}}\otimes\wedge^2\widetilde{E_\mathbf{C}}-\widetilde{E_\mathbf{C}}\otimes
\widetilde{E_\mathbf{C}}\otimes\widetilde{E_\mathbf{C}}+\widetilde{E_\mathbf{C}}
+\wedge^2\widetilde{E_\mathbf{C}}+(2\wedge^2\widetilde{E_\mathbf{C}}-\widetilde{E_\mathbf{C}}\\
&\otimes\widetilde{E_\mathbf{C}}+\widetilde{E_\mathbf{C}})\otimes(2\wedge^2\widetilde{L_{\bf{R}}\otimes\mathbf{C}}-(\widetilde{L_{\bf{R}}\otimes\mathbf{C}})\otimes(\widetilde{L_{\bf{R}}\otimes\mathbf{C}})+\widetilde{L_{\bf{R}}\otimes\mathbf{C}})+\wedge^2\widetilde{L_{\bf{R}}\otimes\mathbf{C}}\\
&\otimes\wedge^2\widetilde{L_{\bf{R}}
\otimes\mathbf{C}}+2\wedge^4\widetilde{L_{\bf{R}}\otimes\mathbf{C}}-2\widetilde{L_{\bf{R}}\otimes\mathbf{C}}\otimes\wedge^3\widetilde{L_{\bf{R}}
\otimes\mathbf{C}}+2\widetilde{L_{\bf{R}}\otimes\mathbf{C}}\otimes\wedge^2\widetilde{L_{\bf{R}}\otimes\mathbf{C}}\\
&-\widetilde{L_{\bf{R}}\otimes\mathbf{C}}\otimes\widetilde{L_{\bf{R}}
\otimes\mathbf{C}}\otimes\widetilde{L_{\bf{R}}
\otimes\mathbf{C}}+\widetilde{L_{\bf{R}}
\otimes\mathbf{C}}+\wedge^2\widetilde{L_{\bf{R}}\otimes\mathbf{C}})]\}^{(7)}=61920\{ e^{\frac{1}{720}c_{2}(W_{i})}\\
    &\cdot{\widehat{A}(TM,\nabla^{TM})}{\rm exp}(\frac{c}{2}){\rm ch}(\triangle(E),g^{\triangle(E)},d) \}^{(7)}.
\end{aligned}
\end{equation}
\end{thm}

\begin{proof}
Direct computation show
\begin{equation}
\begin{aligned}
&\widetilde{Q}(\nabla^{TM},\nabla^{L},P_{i},g,d,\tau)\\
=&\{ e^{\frac{1}{720}c_{2}(W_{i})}{\widehat{A}(TM,\nabla^{TM})}{\rm exp}(\frac{c}{2}){\rm ch}(\triangle(E),g^{\triangle(E)},d) \}^{(4k-1)}+q \{ {\widehat{A}(TM,\nabla^{TM})}\\
&\cdot{\rm exp}(\frac{c}{2}){\rm ch}(\triangle(E),g^{\triangle(E)},d)[-e^{\frac{1}{720}c_{2}(W_{i})}\cdot\tfrac{1}{30}c_{2}(W_{i})+e^{\frac{1}{720}c_{2}(W_{i})}{\rm ch}(-8+W_i)\\
&+e^{\frac{1}{720}c_{2}(W_{i})}{\rm ch}(\widetilde{T_\mathbf{C}M}+\widetilde{E_\mathbf{C}}-\widetilde{E_\mathbf{C}}
    \otimes\widetilde{E_\mathbf{C}}+2\wedge^2\widetilde{E_\mathbf{C}}+\widetilde{L_{\bf{R}}\otimes\mathbf{C}}-(\widetilde{L_{\bf{R}}\otimes\mathbf{C}})\\
    & \otimes(\widetilde{L_{\bf{R}}\otimes\mathbf{C}})+2\wedge^2\widetilde{L_{\bf{R}}\otimes\mathbf{C}})]\}^{(4k-1)}+q^2\{ {\widehat{A}(TM,\nabla^{TM})}{\rm exp}(\frac{c}{2})   \end{aligned}
\end{equation}
\begin{equation}
\begin{aligned}
&\cdot{\rm ch}(\triangle(E),g^{\triangle(E)},d)[e^{\frac{1}{720}c_{2}(W_{i})}\cdot\tfrac{1}{2}(-6+\tfrac{1}{30}c_{2}(W_{i}))\cdot\tfrac{1}{30}c_{2}(W_{i})-e^{\frac{1}{720}c_{2}(W_{i})}\\
&\cdot\tfrac{1}{30}c_{2}(W_{i}){\rm ch}(-8+W_i)+e^{\frac{1}{720}c_{2}(W_{i})}{\rm ch}(20-8W_i+\overline{W_{i}})-e^{\frac{1}{720}c_{2}(W_{i})}\cdot\tfrac{1}{30}c_{2}(W_{i})\\
&\cdot{\rm ch}(\widetilde{T_\mathbf{C}M}+\widetilde{E_\mathbf{C}}-\widetilde{E_\mathbf{C}}
    \otimes\widetilde{E_\mathbf{C}}+2\wedge^2\widetilde{E_\mathbf{C}}+\widetilde{L_{\bf{R}}\otimes\mathbf{C}}-(\widetilde{L_{\bf{R}}\otimes\mathbf{C}})\otimes(\widetilde{L_{\bf{R}}\otimes\mathbf{C}})\\
    &+2\wedge^2\widetilde{L_{\bf{R}}\otimes\mathbf{C}})+e^{\frac{1}{720}c_{2}(W_{i})}{\rm ch}(-8+W_i){\rm ch}(\widetilde{T_\mathbf{C}M}+\widetilde{E_\mathbf{C}}-\widetilde{E_\mathbf{C}}
    \otimes\widetilde{E_\mathbf{C}}+2\wedge^2\widetilde{E_\mathbf{C}}\\
    &+\widetilde{L_{\bf{R}}\otimes\mathbf{C}}-(\widetilde{L_{\bf{R}}\otimes\mathbf{C}})
    \otimes(\widetilde{L_{\bf{R}}\otimes\mathbf{C}})+2\wedge^2\widetilde{L_{\bf{R}}\otimes\mathbf{C}})+e^{\frac{1}{720}c_{2}(W_{i})}{\rm ch}(\widetilde{T_\mathbf{C}M}\\
    &+S^2\widetilde{T_\mathbf{C}M}+\widetilde{T_\mathbf{C}M}\otimes(\widetilde{E_\mathbf{C}}-\widetilde{E_\mathbf{C}}
    \otimes\widetilde{E_\mathbf{C}}+2\wedge^2\widetilde{E_\mathbf{C}}+\widetilde{L_{\bf{R}}\otimes\mathbf{C}}-(\widetilde{L_{\bf{R}}\otimes\mathbf{C}})\\
    &\otimes(\widetilde{L_{\bf{R}}\otimes\mathbf{C}})+2\wedge^2\widetilde{L_{\bf{R}}\otimes\mathbf{C}})+\wedge^2\widetilde{E_\mathbf{C}}\otimes\wedge^2\widetilde{E_\mathbf{C}}
+2\wedge^4\widetilde{E_\mathbf{C}}-2\widetilde{E_\mathbf{C}}\otimes\wedge^3\widetilde{E_\mathbf{C}}\\
&+2\widetilde{E_\mathbf{C}}\otimes\wedge^2\widetilde{E_\mathbf{C}}-\widetilde{E_\mathbf{C}}\otimes
\widetilde{E_\mathbf{C}}\otimes\widetilde{E_\mathbf{C}}+\widetilde{E_\mathbf{C}}
+\wedge^2\widetilde{E_\mathbf{C}}+(2\wedge^2\widetilde{E_\mathbf{C}}-\widetilde{E_\mathbf{C}}
    \otimes\widetilde{E_\mathbf{C}}\\
    &+\widetilde{E_\mathbf{C}})\otimes(2\wedge^2\widetilde{L_{\bf{R}}\otimes\mathbf{C}}-(\widetilde{L_{\bf{R}}\otimes\mathbf{C}})
    \otimes(\widetilde{L_{\bf{R}}\otimes\mathbf{C}})+\widetilde{L_{\bf{R}}\otimes\mathbf{C}})+\wedge^2\widetilde{L_{\bf{R}}\otimes\mathbf{C}}\\
&\otimes\wedge^2\widetilde{L_{\bf{R}}
\otimes\mathbf{C}}+2\wedge^4\widetilde{L_{\bf{R}}\otimes\mathbf{C}}-2\widetilde{L_{\bf{R}}\otimes\mathbf{C}}\otimes\wedge^3\widetilde{L_{\bf{R}}
\otimes\mathbf{C}}
+2\widetilde{L_{\bf{R}}
\otimes\mathbf{C}}\otimes\wedge^2\widetilde{L_{\bf{R}}
\otimes\mathbf{C}}\\
&-\widetilde{L_{\bf{R}}
\otimes\mathbf{C}}\otimes\widetilde{L_{\bf{R}}
\otimes\mathbf{C}}\otimes\widetilde{L_{\bf{R}}
\otimes\mathbf{C}}+\widetilde{L_{\bf{R}}
\otimes\mathbf{C}}+\wedge^2\widetilde{L_{\bf{R}}\otimes\mathbf{C}})]\}^{(4k-1)}+\cdot\cdot\cdot.\nonumber
\end{aligned}
\end{equation}

For a $7$-dimensional manifold, $\widetilde{Q}(\nabla^{TM},\nabla^{L},P_{i},g,d,\tau)$ is a weight $10$ modular form on $\mathrm{SL}(2,\mathbf{Z}).$ Therefore, $\widetilde{Q}(\nabla^{TM},\nabla^{L},P_{i},g,d,\tau)=\lambda E_4(\tau)E_6(\tau).$
We thus obtain the Theorem 3.8.
\end{proof}

One can draw a similar conclusion for an $11$-dimensional manifold.
\begin{thm}
For $11$-dimensional spin manifold, when $3p_1(L)-p_1(M)=0$ and $c_3(E,g,d)=0,$ we conclude that
 \begin{equation}
\begin{aligned}
& \{ {\widehat{A}(TM,\nabla^{TM})}{\rm exp}(\frac{c}{2}){\rm ch}(\triangle(E),g^{\triangle(E)},d)[-e^{\frac{1}{720}c_{2}(W_{i})}\cdot\tfrac{1}{30}c_{2}(W_{i})+e^{\frac{1}{720}c_{2}(W_{i})}\\
&\cdot{\rm ch}(-8+W_i)+e^{\frac{1}{720}c_{2}(W_{i})}{\rm ch}(\widetilde{T_\mathbf{C}M}+\widetilde{E_\mathbf{C}}-\widetilde{E_\mathbf{C}}
    \otimes\widetilde{E_\mathbf{C}}+2\wedge^2\widetilde{E_\mathbf{C}}+\widetilde{L_{\bf{R}}\otimes\mathbf{C}}\\
    &-(\widetilde{L_{\bf{R}}\otimes\mathbf{C}})
    \otimes(\widetilde{L_{\bf{R}}\otimes\mathbf{C}})+2\wedge^2\widetilde{L_{\bf{R}}\otimes\mathbf{C}})]\}^{(11)}=-264\{ e^{\frac{1}{720}c_{2}(W_{i})}{\widehat{A}(TM,\nabla^{TM})}\\
    &\cdot{\rm exp}(\frac{c}{2}){\rm ch}(\triangle(E),g^{\triangle(E)},d) \}^{(11)},
\end{aligned}
\end{equation}
  and
 \begin{equation}
\begin{aligned}
&\{ {\widehat{A}(TM,\nabla^{TM})}{\rm exp}(\frac{c}{2}){\rm ch}(\triangle(E),g^{\triangle(E)},d)[e^{\frac{1}{720}c_{2}(W_{i})}\cdot\tfrac{1}{2}(-6+\tfrac{1}{30}c_{2}(W_{i}))\cdot\tfrac{1}{30}c_{2}(W_{i})\\
&-e^{\frac{1}{720}c_{2}(W_{i})}\cdot\tfrac{1}{30}c_{2}(W_{i}){\rm ch}(-8+W_i)+e^{\frac{1}{720}c_{2}(W_{i})}{\rm ch}(20-8W_i+\overline{W_{i}})-e^{\frac{1}{720}c_{2}(W_{i})}\\
&\cdot\tfrac{1}{30}c_{2}(W_{i}){\rm ch}(\widetilde{T_\mathbf{C}M}+\widetilde{E_\mathbf{C}}-\widetilde{E_\mathbf{C}}
    \otimes\widetilde{E_\mathbf{C}}+2\wedge^2\widetilde{E_\mathbf{C}}+\widetilde{L_{\bf{R}}\otimes\mathbf{C}}-(\widetilde{L_{\bf{R}}\otimes\mathbf{C}})\\
    & \otimes(\widetilde{L_{\bf{R}}\otimes\mathbf{C}})+2\wedge^2\widetilde{L_{\bf{R}}\otimes\mathbf{C}})+e^{\frac{1}{720}c_{2}(W_{i})}{\rm ch}(-8+W_i){\rm ch}(\widetilde{T_\mathbf{C}M}+\widetilde{E_\mathbf{C}}\\
&-\widetilde{E_\mathbf{C}}
    \otimes\widetilde{E_\mathbf{C}}+2\wedge^2\widetilde{E_\mathbf{C}}+\widetilde{L_{\bf{R}}\otimes\mathbf{C}}-(\widetilde{L_{\bf{R}}\otimes\mathbf{C}})
    \otimes(\widetilde{L_{\bf{R}}\otimes\mathbf{C}})+2\wedge^2\widetilde{L_{\bf{R}}\otimes\mathbf{C}}) \end{aligned}
\end{equation}
\begin{equation}
\begin{aligned}
    &+e^{\frac{1}{720}c_{2}(W_{i})}{\rm ch}(\widetilde{T_\mathbf{C}M}+S^2\widetilde{T_\mathbf{C}M}+\widetilde{T_\mathbf{C}M}\otimes(\widetilde{E_\mathbf{C}}-\widetilde{E_\mathbf{C}}
    \otimes\widetilde{E_\mathbf{C}}+2\wedge^2\widetilde{E_\mathbf{C}}+\widetilde{L_{\bf{R}}\otimes\mathbf{C}}\\
    &-(\widetilde{L_{\bf{R}}\otimes\mathbf{C}})
    \otimes(\widetilde{L_{\bf{R}}\otimes\mathbf{C}})+2\wedge^2\widetilde{L_{\bf{R}}\otimes\mathbf{C}})+\wedge^2\widetilde{E_\mathbf{C}}\otimes\wedge^2\widetilde{E_\mathbf{C}}
+2\wedge^4\widetilde{E_\mathbf{C}}-2\widetilde{E_\mathbf{C}}\\
&\otimes\wedge^3\widetilde{E_\mathbf{C}}
+2\widetilde{E_\mathbf{C}}\otimes\wedge^2\widetilde{E_\mathbf{C}}-\widetilde{E_\mathbf{C}}\otimes
\widetilde{E_\mathbf{C}}\otimes\widetilde{E_\mathbf{C}}+\widetilde{E_\mathbf{C}}
+\wedge^2\widetilde{E_\mathbf{C}}+(2\wedge^2\widetilde{E_\mathbf{C}}-\widetilde{E_\mathbf{C}}\\
&\otimes\widetilde{E_\mathbf{C}}+\widetilde{E_\mathbf{C}})\otimes(2\wedge^2\widetilde{L_{\bf{R}}\otimes\mathbf{C}}-(\widetilde{L_{\bf{R}}\otimes\mathbf{C}})\otimes(\widetilde{L_{\bf{R}}\otimes\mathbf{C}})+\widetilde{L_{\bf{R}}\otimes\mathbf{C}})+\wedge^2\widetilde{L_{\bf{R}}\otimes\mathbf{C}}\\
&\otimes\wedge^2\widetilde{L_{\bf{R}}
\otimes\mathbf{C}}+2\wedge^4\widetilde{L_{\bf{R}}\otimes\mathbf{C}}-2\widetilde{L_{\bf{R}}\otimes\mathbf{C}}\otimes\wedge^3\widetilde{L_{\bf{R}}
\otimes\mathbf{C}}+2\widetilde{L_{\bf{R}}\otimes\mathbf{C}}\otimes\wedge^2\widetilde{L_{\bf{R}}\otimes\mathbf{C}}\\
&-\widetilde{L_{\bf{R}}\otimes\mathbf{C}}\otimes\widetilde{L_{\bf{R}}
\otimes\mathbf{C}}\otimes\widetilde{L_{\bf{R}}
\otimes\mathbf{C}}+\widetilde{L_{\bf{R}}
\otimes\mathbf{C}}+\wedge^2\widetilde{L_{\bf{R}}\otimes\mathbf{C}})]\}^{(11)}=-135432\\
    &\cdot\{ e^{\frac{1}{720}c_{2}(W_{i})}{\widehat{A}(TM,\nabla^{TM})}{\rm exp}(\frac{c}{2}){\rm ch}(\triangle(E),g^{\triangle(E)},d) \}^{(11)}.\nonumber
\end{aligned}
\end{equation}
\end{thm}

\subsection{Modular forms and Witten genus in $(4k+1)$-dimensional spin$^c$ manifolds}\hfill\\

Let $M$ be a closed oriented ${\rm spin}^{c}$ manifold of dimension $4k+1$.

Choose $\Theta^*(T_{\mathbf{C}}M,L_{\bf{R}}\otimes\mathbf{C})$ be the virtual complex vector bundle over $M$ defined by
\begin{equation}
\begin{aligned}
\Theta^*(T_{\mathbf{C}}M,L_{\bf{R}}\otimes\mathbf{C})=\bigotimes _{n=1}^{\infty}S_{q^n}(\widetilde{T_{\mathbf{C}}M})\otimes
\bigotimes _{m=1}^{\infty}\wedge_{-q^m}(\widetilde{L_{\bf{R}}\otimes\mathbf{C}}).
\end{aligned}
\end{equation}

For $c=0$, the Witten bundle is given by $\Theta^*(T_{\mathbf{C}}M, L_{\mathbf{R}}\otimes \mathbf{C}).$
With this bundle, we define the generalized Witten form by
\begin{equation}
\begin{aligned}
\bar{Q}(\nabla^{TM},\nabla^{L},P_{i},g,d,\tau)&=\{e^{\frac{1}{24}E_{2}(\tau)\cdot\frac{1}{30}c_{2}(W_{i})} \widehat{A}(TM,\nabla^{TM}){\rm ch}(\Theta^*(T_{\mathbf{C}}M,L_{\bf{R}}\otimes\mathbf{C}))\\
&\cdot{\rm exp}(\frac{c}{2}){\rm ch}(Q(E),g^{Q(E)},d,\tau)\varphi(\tau)^{(8)}ch(\mathcal{V}_{i})\}^{(4k+1)}.
\end{aligned}
\end{equation}

Based on this definition and
\begin{equation}
\begin{aligned}
&\bar{Q}(\nabla^{TM},\nabla^{L},P_{i},g,d,\tau)\\
=&\{ e^{\frac{1}{720}c_{2}(W_{i})}{\widehat{A}(TM,\nabla^{TM})}{\rm exp}(\frac{c}{2}){\rm ch}(\triangle(E),g^{\triangle(E)},d) \}^{(4k+1)}+q \{ {\widehat{A}(TM,\nabla^{TM})}\\
&\cdot{\rm exp}(\frac{c}{2}){\rm ch}(\triangle(E),g^{\triangle(E)},d)[-e^{\frac{1}{720}c_{2}(W_{i})}\cdot\tfrac{1}{30}c_{2}(W_{i})+e^{\frac{1}{720}c_{2}(W_{i})}{\rm ch}(-8+W_i)\\
&+e^{\frac{1}{720}c_{2}(W_{i})}{\rm ch}(\widetilde{T_\mathbf{C}M}+\widetilde{E_\mathbf{C}}-\widetilde{E_\mathbf{C}}
    \otimes\widetilde{E_\mathbf{C}}+2\wedge^2\widetilde{E_\mathbf{C}}-\widetilde{L_{\bf{R}}\otimes\mathbf{C}})]\}^{(4k+1)}\\
    &+q^2\{ {\widehat{A}(TM,\nabla^{TM})}{\rm exp}(\frac{c}{2}){\rm ch}(\triangle(E),g^{\triangle(E)},d)[e^{\frac{1}{720}c_{2}(W_{i})}\cdot\tfrac{1}{2}(-6+\tfrac{1}{30}c_{2}(W_{i}))\\
    &\cdot\tfrac{1}{30}c_{2}(W_{i})-e^{\frac{1}{720}c_{2}(W_{i})}\cdot\tfrac{1}{30}c_{2}(W_{i}){\rm ch}(-8+W_i)+e^{\frac{1}{720}c_{2}(W_{i})}{\rm ch}(20-8W_i+\overline{W_{i}}) \end{aligned}
\end{equation}
\begin{equation}
\begin{aligned}
    &-e^{\frac{1}{720}c_{2}(W_{i})}\cdot\tfrac{1}{30}c_{2}(W_{i}){\rm ch}(\widetilde{T_\mathbf{C}M}+\widetilde{E_\mathbf{C}}-\widetilde{E_\mathbf{C}}
    \otimes\widetilde{E_\mathbf{C}}+2\wedge^2\widetilde{E_\mathbf{C}}-\widetilde{L_{\bf{R}}\otimes\mathbf{C}})\\
    &+e^{\frac{1}{720}c_{2}(W_{i})}{\rm ch}(-8+W_i){\rm ch}(\widetilde{T_\mathbf{C}M}+\widetilde{E_\mathbf{C}}-\widetilde{E_\mathbf{C}}
    \otimes\widetilde{E_\mathbf{C}}+2\wedge^2\widetilde{E_\mathbf{C}}-\widetilde{L_{\bf{R}}\otimes\mathbf{C}})\\
    &+e^{\frac{1}{720}c_{2}(W_{i})}{\rm ch}(\widetilde{T_\mathbf{C}M}+S^2\widetilde{T_\mathbf{C}M}+\widetilde{T_\mathbf{C}M}\otimes(\widetilde{E_\mathbf{C}}-\widetilde{E_\mathbf{C}}
    \otimes\widetilde{E_\mathbf{C}}+2\wedge^2\widetilde{E_\mathbf{C}}\\
    &-\widetilde{L_{\bf{R}}\otimes\mathbf{C}})-\widetilde{L_{\bf{R}}\otimes\mathbf{C}}+\wedge^2\widetilde{L_{\bf{R}}\otimes\mathbf{C}}-\widetilde{L_{\bf{R}}\otimes\mathbf{C}} \otimes(2\wedge^2\widetilde{E_\mathbf{C}}-\widetilde{E_\mathbf{C}}
    \otimes\widetilde{E_\mathbf{C}}+\widetilde{E_\mathbf{C}})\\
    &+\wedge^2\widetilde{E_\mathbf{C}}\otimes\wedge^2\widetilde{E_\mathbf{C}}
+2\wedge^4\widetilde{E_\mathbf{C}}-2\widetilde{E_\mathbf{C}}\otimes\wedge^3\widetilde{E_\mathbf{C}}
+2\widetilde{E_\mathbf{C}}\otimes\wedge^2\widetilde{E_\mathbf{C}}-\widetilde{E_\mathbf{C}}\otimes
\widetilde{E_\mathbf{C}}\otimes\widetilde{E_\mathbf{C}}\\
&+\widetilde{E_\mathbf{C}}+\wedge^2\widetilde{E_\mathbf{C}})]\}^{(4k+1)}+\cdot\cdot\cdot,\nonumber
\end{aligned}
\end{equation}
 we obtain the following three theorems.

\begin{thm}
Let ${\rm dim}M=4k+1,$ for $k\leq 3.$ If $p_1(L)-p_1(M)=0$ and $c_3(E,g,d)=0,$ then $\bar{Q}(\nabla^{TM},\nabla^{L},P_{i},g,d,\tau)$ is a modular form over $\mathrm{SL}(2,\mathbf{Z})$ with the weight $2k+4.$
\end{thm}

\begin{thm}
For $9$-dimensional spin manifold, when $p_1(L)-p_1(M)=0$ and $c_3(E,g,d)=0,$ we conclude that
 \begin{equation}
\begin{aligned}
&\{ {\widehat{A}(TM,\nabla^{TM})}{\rm exp}(\frac{c}{2}){\rm ch}(\triangle(E),g^{\triangle(E)},d)[-e^{\frac{1}{720}c_{2}(W_{i})}\cdot\tfrac{1}{30}c_{2}(W_{i})\\
&+e^{\frac{1}{720}c_{2}(W_{i})}{\rm ch}(-8+W_i)+e^{\frac{1}{720}c_{2}(W_{i})}{\rm ch}(\widetilde{T_\mathbf{C}M}+\widetilde{E_\mathbf{C}}-\widetilde{E_\mathbf{C}}
    \otimes\widetilde{E_\mathbf{C}}\\
    &+2\wedge^2\widetilde{E_\mathbf{C}}-\widetilde{L_{\bf{R}}\otimes\mathbf{C}})]\}^{(9)}=480\{ e^{\frac{1}{720}c_{2}(W_{i})}{\widehat{A}(TM,\nabla^{TM})}{\rm exp}(\frac{c}{2})\\
    &\cdot{\rm ch}(\triangle(E),g^{\triangle(E)},d) \}^{(9)},
\end{aligned}
\end{equation}
  and
 \begin{equation}
\begin{aligned}
&\{ {\widehat{A}(TM,\nabla^{TM})}{\rm exp}(\frac{c}{2}){\rm ch}(\triangle(E),g^{\triangle(E)},d)[e^{\frac{1}{720}c_{2}(W_{i})}\cdot\tfrac{1}{2}(-6+\tfrac{1}{30}c_{2}(W_{i}))\\
    &\cdot\tfrac{1}{30}c_{2}(W_{i})-e^{\frac{1}{720}c_{2}(W_{i})}\cdot\tfrac{1}{30}c_{2}(W_{i}){\rm ch}(-8+W_i)+e^{\frac{1}{720}c_{2}(W_{i})}{\rm ch}(20-8W_i+\overline{W_{i}})\\
    &-e^{\frac{1}{720}c_{2}(W_{i})}\cdot\tfrac{1}{30}c_{2}(W_{i}){\rm ch}(\widetilde{T_\mathbf{C}M}+\widetilde{E_\mathbf{C}}-\widetilde{E_\mathbf{C}}
    \otimes\widetilde{E_\mathbf{C}}+2\wedge^2\widetilde{E_\mathbf{C}}-\widetilde{L_{\bf{R}}\otimes\mathbf{C}})\\
    &+e^{\frac{1}{720}c_{2}(W_{i})}{\rm ch}(-8+W_i){\rm ch}(\widetilde{T_\mathbf{C}M}+\widetilde{E_\mathbf{C}}-\widetilde{E_\mathbf{C}}
    \otimes\widetilde{E_\mathbf{C}}+2\wedge^2\widetilde{E_\mathbf{C}}-\widetilde{L_{\bf{R}}\otimes\mathbf{C}})\\
    &+e^{\frac{1}{720}c_{2}(W_{i})}{\rm ch}(\widetilde{T_\mathbf{C}M}+S^2\widetilde{T_\mathbf{C}M}+\widetilde{T_\mathbf{C}M}\otimes(\widetilde{E_\mathbf{C}}-\widetilde{E_\mathbf{C}}
    \otimes\widetilde{E_\mathbf{C}}+2\wedge^2\widetilde{E_\mathbf{C}}\\
    &-\widetilde{L_{\bf{R}}\otimes\mathbf{C}})-\widetilde{L_{\bf{R}}\otimes\mathbf{C}}+\wedge^2\widetilde{L_{\bf{R}}\otimes\mathbf{C}}-\widetilde{L_{\bf{R}}\otimes\mathbf{C}} \otimes(2\wedge^2\widetilde{E_\mathbf{C}}-\widetilde{E_\mathbf{C}}
    \otimes\widetilde{E_\mathbf{C}}+\widetilde{E_\mathbf{C}})\\
    &+\wedge^2\widetilde{E_\mathbf{C}}\otimes\wedge^2\widetilde{E_\mathbf{C}}
+2\wedge^4\widetilde{E_\mathbf{C}}-2\widetilde{E_\mathbf{C}}\otimes\wedge^3\widetilde{E_\mathbf{C}}
+2\widetilde{E_\mathbf{C}}\otimes\wedge^2\widetilde{E_\mathbf{C}}-\widetilde{E_\mathbf{C}}\otimes
\widetilde{E_\mathbf{C}}\otimes\widetilde{E_\mathbf{C}}\\
&+\widetilde{E_\mathbf{C}}+\wedge^2\widetilde{E_\mathbf{C}})]\}^{(9)}=61920\{ e^{\frac{1}{720}c_{2}(W_{i})}{\widehat{A}(TM,\nabla^{TM})}{\rm exp}(\frac{c}{2})\\
    &\cdot{\rm ch}(\triangle(E),g^{\triangle(E)},d) \}^{(9)}.
\end{aligned}
\end{equation}
\end{thm}

\begin{thm}
For $13$-dimensional spin manifold, when $p_1(L)-p_1(M)=0$ and $c_3(E,g,d)=0,$ we conclude that
 \begin{equation}
\begin{aligned}
&\{ {\widehat{A}(TM,\nabla^{TM})}{\rm exp}(\frac{c}{2}){\rm ch}(\triangle(E),g^{\triangle(E)},d)[-e^{\frac{1}{720}c_{2}(W_{i})}\cdot\tfrac{1}{30}c_{2}(W_{i})\\
&+e^{\frac{1}{720}c_{2}(W_{i})}{\rm ch}(-8+W_i)+e^{\frac{1}{720}c_{2}(W_{i})}{\rm ch}(\widetilde{T_\mathbf{C}M}+\widetilde{E_\mathbf{C}}-\widetilde{E_\mathbf{C}}
    \otimes\widetilde{E_\mathbf{C}}\\
    &+2\wedge^2\widetilde{E_\mathbf{C}}-\widetilde{L_{\bf{R}}\otimes\mathbf{C}})]\}^{(13)}=-264\{ e^{\frac{1}{720}c_{2}(W_{i})}{\widehat{A}(TM,\nabla^{TM})}{\rm exp}(\frac{c}{2})\\
    &\cdot{\rm ch}(\triangle(E),g^{\triangle(E)},d) \}^{(13)},
\end{aligned}
\end{equation}
  and
 \begin{equation}
\begin{aligned}
&\{ {\widehat{A}(TM,\nabla^{TM})}{\rm exp}(\frac{c}{2}){\rm ch}(\triangle(E),g^{\triangle(E)},d)[e^{\frac{1}{720}c_{2}(W_{i})}\cdot\tfrac{1}{2}(-6+\tfrac{1}{30}c_{2}(W_{i}))\\
    &\cdot\tfrac{1}{30}c_{2}(W_{i})-e^{\frac{1}{720}c_{2}(W_{i})}\cdot\tfrac{1}{30}c_{2}(W_{i}){\rm ch}(-8+W_i)+e^{\frac{1}{720}c_{2}(W_{i})}{\rm ch}(20-8W_i+\overline{W_{i}})\\
    &-e^{\frac{1}{720}c_{2}(W_{i})}\cdot\tfrac{1}{30}c_{2}(W_{i}){\rm ch}(\widetilde{T_\mathbf{C}M}+\widetilde{E_\mathbf{C}}-\widetilde{E_\mathbf{C}}
    \otimes\widetilde{E_\mathbf{C}}+2\wedge^2\widetilde{E_\mathbf{C}}-\widetilde{L_{\bf{R}}\otimes\mathbf{C}})\\
    &+e^{\frac{1}{720}c_{2}(W_{i})}{\rm ch}(-8+W_i){\rm ch}(\widetilde{T_\mathbf{C}M}+\widetilde{E_\mathbf{C}}-\widetilde{E_\mathbf{C}}
    \otimes\widetilde{E_\mathbf{C}}+2\wedge^2\widetilde{E_\mathbf{C}}-\widetilde{L_{\bf{R}}\otimes\mathbf{C}})\\
    &+e^{\frac{1}{720}c_{2}(W_{i})}{\rm ch}(\widetilde{T_\mathbf{C}M}+S^2\widetilde{T_\mathbf{C}M}+\widetilde{T_\mathbf{C}M}\otimes(\widetilde{E_\mathbf{C}}-\widetilde{E_\mathbf{C}}
    \otimes\widetilde{E_\mathbf{C}}+2\wedge^2\widetilde{E_\mathbf{C}}\\
    &-\widetilde{L_{\bf{R}}\otimes\mathbf{C}})-\widetilde{L_{\bf{R}}\otimes\mathbf{C}}+\wedge^2\widetilde{L_{\bf{R}}\otimes\mathbf{C}}-\widetilde{L_{\bf{R}}\otimes\mathbf{C}} \otimes(2\wedge^2\widetilde{E_\mathbf{C}}-\widetilde{E_\mathbf{C}}
    \otimes\widetilde{E_\mathbf{C}}+\widetilde{E_\mathbf{C}})\\
    &+\wedge^2\widetilde{E_\mathbf{C}}\otimes\wedge^2\widetilde{E_\mathbf{C}}
+2\wedge^4\widetilde{E_\mathbf{C}}-2\widetilde{E_\mathbf{C}}\otimes\wedge^3\widetilde{E_\mathbf{C}}
+2\widetilde{E_\mathbf{C}}\otimes\wedge^2\widetilde{E_\mathbf{C}}-\widetilde{E_\mathbf{C}}\otimes
\widetilde{E_\mathbf{C}}\otimes\widetilde{E_\mathbf{C}}\\
&+\widetilde{E_\mathbf{C}}+\wedge^2\widetilde{E_\mathbf{C}})]\}^{(13)}=-135432\{ e^{\frac{1}{720}c_{2}(W_{i})}{\widehat{A}(TM,\nabla^{TM})}{\rm exp}(\frac{c}{2})\\
    &\cdot{\rm ch}(\triangle(E),g^{\triangle(E)},d) \}^{(13)}.
\end{aligned}
\end{equation}
\end{thm}

\section{Modular forms and Witten genus associated with $E_{8}\times E_{8}$ bundles}

In this section, we construct modular forms over $SL(2,\mathbf{Z})$ associated with $E_{8}\times E_{8}$ bundles and extend this construction to spin and spin$^c$ manifolds, this leads to the systematic derivation of a new series of cancellation formulas.

\subsection{Modular forms in $(4k-1)$-dimensional spin manifolds}\hfill\\

Set $M$ be a $(4k-1)$-dimensional spin manifold and $\triangle(M)$ be the spinor bundle.
We generalize $Q(\nabla^{TM}, P_{i}, g, d, \tau)$ to the case of $E_{8} \times E_{8}$ bundles and redefine it as follows:
\begin{align}
Q(\nabla^{TM},P_{i},P_{j},g,d,\tau)=&\{e^{\frac{1}{24}E_{2}(\tau)\cdot\frac{1}{30}(c_{2}(W_{i})+c_{2}(W_{j}))} \widehat{A}(TM,\nabla^{TM}){\rm ch}([\triangle(M)\\
&\otimes \Theta_1(T_{C}M)+2^{2k-1}\Theta_2(T_{C}M)+2^{2k-1}\Theta_3(T_{C}M)])\nonumber\\
&{\rm ch}(Q(E),g^{Q(E)},d,\tau)\cdot\varphi(\tau)^{(16)}ch(\mathcal{V}_{i})ch(\mathcal{V}_{j})\}^{(4k-1)}.\nonumber
\end{align}

\begin{thm}
Let ${\rm dim}M=4k-1,$ for $k\leq 4.$ Suppose $g$ has a lift to the spin group ${\rm Spin}(N):g^{\Delta}:M\to {\rm Spin}(N)$. If $c_3(E,g,d)=0$, then $Q(\nabla^{TM},P_{i},P_{j},g,d,\tau)$ is a modular form over $SL(2,{\bf Z})$ with the weight $2k+8$.
\end{thm}
\begin{proof}
Choose
\begin{align}
Q(\nabla^{TM},P_{i},P_{j},g,d,\tau)=&Q(M,P_{i},P_{j},\tau)\cdot{\rm ch}(Q(E),g^{Q(E)},d,\tau)^{(4r-1)},
\end{align}
in this way
\begin{align}
Q(M,P_{i},P_{j},\tau)=&\{e^{\frac{1}{24}E_{2}(\tau)\cdot\frac{1}{30}(c_{2}(W_{i})+c_{2}(W_{j}))} \widehat{A}(TM,\nabla^{TM}){\rm ch}([\triangle(M)\\
&\otimes \Theta_1(T_{C}M)+2^{2k-1}\Theta_2(T_{C}M)+2^{2k-1}\Theta_3(T_{C}M)])\nonumber\\
&\cdot\varphi(\tau)^{(16)}ch(\mathcal{V}_{i})ch(\mathcal{V}_{j})\}^{(4p)}\nonumber
\end{align}
So
\begin{equation}
\begin{aligned}
Q(M,P_{i},P_{j},\tau) = & \left\{ e^{\frac{1}{24}E_{2}(\tau)\cdot\frac{1}{30}(c_{2}(W_{i})+c_{2}(W_{j}))} \cdot
\prod_{j=1}^{2k-1}\frac{2 x_j\theta'(0,\tau)}{\theta(x_j,\tau)}
\Big(
\frac{\theta_1(x_j,\tau)}{\theta_1(0,\tau)}
+ \frac{\theta_2(x_j,\tau)}{\theta_2(0,\tau)}\right.  \\
& \left. +\frac{\theta_3(x_j,\tau)}{\theta_3(0,\tau)}\Big)\cdot \frac{1}{2}\Big(\prod_{l=1}^{8}\theta_{1}(y_{l}^{i}, \tau)+\prod_{l=1}^{8}\theta_{2}(y_{l}^{i}, \tau)+\prod_{l=1}^{8}\theta_{3}(y_{l}^{i}, \tau)\Big) \vphantom{\prod_{j=1}^{2k-1}}\right.  \\
& \left. \cdot \frac{1}{2}\Big(\prod_{l=1}^{8}\theta_{1}(y_{l}^{j}, \tau)+\prod_{l=1}^{8}\theta_{2}(y_{l}^{j}, \tau)+\prod_{l=1}^{8}\theta_{3}(y_{l}^{j}, \tau)\Big) \vphantom{\prod_{j=1}^{2k-1}} \right\}^{(4p)}.
    \end{aligned}
  \end{equation}
We check at once that
\begin{align}
&Q(M, P_{i},P_{j},\tau+1) =Q(M, \tau),\\
&Q\Big(M, P_{i},P_{j},-\frac{1}{\tau}\Big) =\tau^{2p+8}Q(M, \tau).
\end{align}
This clearly forces
\begin{align}
&Q(\nabla^{TM},P_{i},P_{j},g,d,\tau+1)=Q(\nabla^{TM},P_{i},g,d,\tau),\\
&Q\Big(\nabla^{TM},P_{i},P_{j},g,d,-\frac{1}{\tau}\Big) =\tau^{2r+2p+8}Q(\nabla^{TM},P_{i},g,d,\tau)=\tau^{2k+8}Q(\nabla^{TM},P_{i},g,d,\tau),
\end{align}
where $c_3(E_C,g,d)=0.$
So $Q(\nabla^{TM},P_{i},P_{j},g,d,\tau)$ is a modular form over $SL(2,{\bf Z})$ with the weight $2r+2p+8=2k+8$.
\end{proof}

\begin{thm}
For $7$-dimensional spin manifold, when $c_3(E,g,d)=0,$ we conclude that
  \begin{equation}
    \begin{aligned}
&\{e^{\frac{1}{720}(c_{2}(W_{i})+c_{2}(W_{j}))}\widehat{A}(TM,\nabla^{TM})[{\rm ch}(\triangle(M)\otimes(2\widetilde{T_\mathbf{C}M}+\wedge^2\widetilde{T_\mathbf{C}M}
+\widetilde{T_\mathbf{C}M}\\
&\otimes\widetilde{T_\mathbf{C}M}+S^2\widetilde{T_\mathbf{C}M})){\rm ch}(\triangle(E),g^{\triangle(E)},d)+{\rm ch}(\triangle(M)\otimes2\widetilde{T_{\mathbf{C}}M}){\rm ch}(\triangle(E)\\
&\otimes(2\wedge^2\widetilde{E_\mathbf{C}}-\widetilde{E_\mathbf{C}}\otimes\widetilde{E_\mathbf{C}}+\widetilde{E_\mathbf{C}}),g,d)+{\rm ch}(\triangle(M)){\rm ch}(\triangle(E)\otimes(\wedge^2\widetilde{E_\mathbf{C}}\\
&\otimes\wedge^2\widetilde{E_\mathbf{C}}
+2\wedge^4\widetilde{E_\mathbf{C}}-2\widetilde{E_\mathbf{C}}\otimes\wedge^3\widetilde{E_\mathbf{C}}
+2\widetilde{E_\mathbf{C}}\otimes\wedge^2\widetilde{E_\mathbf{C}}-\widetilde{E_\mathbf{C}}\otimes\widetilde{E_\mathbf{C}}\otimes\widetilde{E_\mathbf{C}}\\
&+\widetilde{E_\mathbf{C}}+\wedge^2\widetilde{E_\mathbf{C}}),g,d)]+8e^{\frac{1}{720}(c_{2}(W_{i})+c_{2}(W_{j}))}\widehat{A}(TM,\nabla^{TM})[{\rm ch}(\wedge^4\widetilde{T_\mathbf{C}M}\\
&+\wedge^2\widetilde{T_\mathbf{C}M}\otimes\widetilde{T_\mathbf{C}M}+\widetilde{T_\mathbf{C}M}\otimes\widetilde{T_\mathbf{C}M}+S^2\widetilde{T_\mathbf{C}M}+\widetilde{T_\mathbf{C}M}){\rm ch}(\triangle(E),g^{\triangle(E)},d)\\
&+{\rm ch}(\widetilde{T_{\mathbf{C}}M}+\wedge^2\widetilde{T_{\mathbf{C}}M}){\rm ch}(\triangle(E)\otimes(2\wedge^2\widetilde{E_\mathbf{C}}-\widetilde
{E_\mathbf{C}}\otimes\widetilde{E_\mathbf{C}}+\widetilde{E_\mathbf{C}}),g,d)\\
&+{\rm ch}(\triangle(E)\otimes(\wedge^2\widetilde{E_\mathbf{C}}\otimes\wedge^2
\widetilde{E_\mathbf{C}}
+2\wedge^4\widetilde{E_\mathbf{C}}-2\widetilde{E_\mathbf{C}}\otimes\wedge^3
\widetilde{E_\mathbf{C}}
+2\widetilde{E_\mathbf{C}}\otimes\wedge^2\widetilde{E_\mathbf{C}}\\
&-\widetilde{E_\mathbf{C}}\otimes\widetilde{E_\mathbf{C}}\otimes\widetilde{E_\mathbf{C}}
+\widetilde{E_\mathbf{C}}+\wedge^2\widetilde{E_\mathbf{C}}),g,d)]+e^{\frac{1}{720}(c_{2}(W_{i})+c_{2}(W_{j}))}[{\rm ch}(W_i+W_j-16)\\
&-\tfrac{1}{30}(c_{2}(W_{i})+c_{2}(W_{j}))][\widehat{A}(TM,\nabla^{TM}){\rm ch}(\triangle(M)\otimes 2\widetilde{T_{\mathbf{C}}M}){\rm ch}(\triangle(E),g^{\triangle(E)},d)\\
&+\widehat{A}(TM,\nabla^{TM}){\rm ch}(\triangle(M)){\rm ch}(\Delta(E)\otimes(2\wedge^2\widetilde{E_\mathbf{C}}-\widetilde{E_\mathbf{C}}\otimes\widetilde{E_\mathbf{C}}+\widetilde{E_\mathbf{C}}),g,d)\\
&+8\widehat{A}(TM,\nabla^{TM}){\rm ch}(\widetilde{T_{\mathbf{C}}M}+\wedge^2\widetilde{T_{\mathbf{C}}M}){\rm ch}(\triangle(E),g^{\triangle(E)},d)+8\widehat{A}(TM,\nabla^{TM})\\
&\cdot{\rm ch}(\Delta(E)\otimes(2\wedge^2\widetilde{E_\mathbf{C}}-\widetilde{E_\mathbf{C}}
    \otimes\widetilde{E_\mathbf{C}}+\widetilde{E_\mathbf{C}}),g,d)]+e^{\frac{1}{720}(c_{2}(W_{i})+c_{2}(W_{j}))}\\
    &\cdot[{\rm ch}(-16W_i-16W_j+W_iW_j+\overline{W_{i}}+\overline{W_{j}}+104)-\tfrac{1}{30}(c_{2}(W_{i})+c_{2}(W_{j}))\\
    &\cdot{\rm ch}(W_i+W_j-16)+\tfrac{1}{2}(-6+\tfrac{1}{30}(c_{2}(W_{i})+c_{2}(W_{j})))\cdot\tfrac{1}{30}(c_{2}(W_{i})+c_{2}(W_{j}))]\\
    &\cdot[{\widehat{A}(TM,\nabla^{TM})}{\rm ch}(\triangle(M)){\rm ch}(\triangle(E),g^{\triangle(E)},d) +8{\widehat{A}(TM,\nabla^{TM})}\\
    &\cdot{\rm ch}(\triangle(E),g^{\triangle(E)},d)]\}^{(7)}=196560 \{ e^{\frac{1}{720}(c_{2}(W_{i})+c_{2}(W_{j}))}[{\widehat{A}(TM,\nabla^{TM})}{\rm ch}(\triangle(M))\\
    &\cdot{\rm ch}(\triangle(E),g^{\triangle(E)},d) +8{\widehat{A}(TM,\nabla^{TM})}{\rm ch}(\triangle(E),g^{\triangle(E)},d)]\}^{(7)}-24\{ e^{\frac{1}{720}c_{2}(W_{i})}\\
    &\cdot e^{\frac{1}{720}c_{2}(W_{j})}{\widehat{A}(TM,\nabla^{TM})}[{\rm ch}(\triangle(M)\otimes 2\widetilde{T_{\mathbf{C}}M}){\rm ch}(\triangle(E),g^{\triangle(E)},d)+{\rm ch}(\triangle(M))\\
    &\cdot{\rm ch}(\Delta(E)\otimes(2\wedge^2\widetilde{E_\mathbf{C}}-\widetilde{E_\mathbf{C}}
    \otimes\widetilde{E_\mathbf{C}}+\widetilde{E_\mathbf{C}}),g,d)]+8e^{\frac{1}{720}(c_{2}(W_{i})+c_{2}(W_{j}))}\\
    &\cdot\widehat{A}(TM,\nabla^{TM})[{\rm ch}(\widetilde{T_{\mathbf{C}}M}+\wedge^2\widetilde{T_{\mathbf{C}}M}){\rm ch}(\triangle(E),g^{\triangle(E)},d)+{\rm ch}(\Delta(E)\otimes(2\wedge^2\widetilde{E_\mathbf{C}}\\
    &-\widetilde{E_\mathbf{C}} \otimes\widetilde{E_\mathbf{C}}+\widetilde{E_\mathbf{C}}),g,d)]+e^{\frac{1}{720}(c_{2}(W_{i})+c_{2}(W_{j}))}[{\rm ch}(W_i+W_j-16)-\tfrac{1}{30}(c_{2}(W_{i})\\
    &+c_{2}(W_{j}))][{\widehat{A}(TM,\nabla^{TM})}{\rm ch}(\triangle(M)){\rm ch}(\triangle(E),g^{\triangle(E)},d)+8{\widehat{A}(TM,\nabla^{TM})}\\
    &\cdot{\rm ch}(\triangle(E),g^{\triangle(E)},d)]\}^{(7)}.
    \end{aligned}
  \end{equation}
\end{thm}

\begin{proof}
An easy computation shows that
\begin{equation}
\begin{aligned}
e^{\frac{1}{24}E_{2}(\tau)\cdot\frac{1}{30}(c_{2}(W_{i})+c_{2}(W_{j}))}=&e^{\frac{1}{720}(c_{2}(W_{i})+c_{2}(W_{j}))}\\
&-qe^{\frac{1}{720}(c_{2}(W_{i})+c_{2}(W_{j}))}\cdot\tfrac{1}{30}(c_{2}(W_{i})+c_{2}(W_{j})) 
\end{aligned}
\end{equation}
\begin{equation}
\begin{aligned}
&+q^2e^{\frac{1}{720}(c_{2}(W_{i})+c_{2}(W_{j}))}\cdot\tfrac{1}{2}(-6+\tfrac{1}{30}(c_{2}(W_{i})+c_{2}(W_{j})))\\
&\cdot\tfrac{1}{30}(c_{2}(W_{i})+c_{2}(W_{j}))+\cdot\cdot\cdot,\nonumber
\end{aligned}
\end{equation}
\begin{equation}
\begin{aligned}
\varphi(\tau)^{(16)}ch(\mathcal{V}_{i})ch(\mathcal{V}_{j})&=1+q{\rm ch}(W_i+W_j-16)\\
&+q^2{\rm ch}(-16W_i-16W_j+W_iW_j+\overline{W_{i}}+\overline{W_{j}}+104)+\cdot\cdot\cdot.
\end{aligned}
\end{equation}
Therefore
\begin{equation}
\begin{aligned}
&Q(\nabla^{TM},P_{i},P_{j},g,d,\tau)\\
=&\{ e^{\frac{1}{720}(c_{2}(W_{i})+c_{2}(W_{j}))}[{\widehat{A}(TM,\nabla^{TM})}{\rm ch}(\triangle(M)){\rm ch}(\triangle(E),g^{\triangle(E)},d) +2^{2k-1}\\
&\cdot{\widehat{A}(TM,\nabla^{TM})}{\rm ch}(\triangle(E),g^{\triangle(E)},d)]\}^{(4k-1)}+q \{ e^{\frac{1}{720}(c_{2}(W_{i})+c_{2}(W_{j}))}{\widehat{A}(TM,\nabla^{TM})}\\
&\cdot[{\rm ch}(\triangle(M)\otimes2\widetilde{T_{\mathbf{C}}M}){\rm ch}(\triangle(E),g^{\triangle(E)},d)+{\rm ch}(\triangle(M)){\rm ch}(\Delta(E)\otimes(2\wedge^2\widetilde{E_\mathbf{C}}\\
&-\widetilde{E_\mathbf{C}}
    \otimes\widetilde{E_\mathbf{C}}+\widetilde{E_\mathbf{C}}),g,d)]+2^{2k-1}e^{\frac{1}{720}(c_{2}(W_{i})+c_{2}(W_{j}))}\widehat{A}(TM,\nabla^{TM})[{\rm ch}(\widetilde{T_{\mathbf{C}}M}\\
    &+\wedge^2\widetilde{T_{\mathbf{C}}M}){\rm ch}(\triangle(E),g^{\triangle(E)},d)+{\rm ch}(\Delta(E)\otimes(2\wedge^2\widetilde{E_\mathbf{C}}-\widetilde{E_\mathbf{C}}\otimes\widetilde{E_\mathbf{C}}+\widetilde{E_\mathbf{C}}),g,d)]\\
    &+e^{\frac{1}{720}(c_{2}(W_{i})+c_{2}(W_{j}))}[{\rm ch}(W_i+W_j-16)-\tfrac{1}{30}(c_{2}(W_{i})+c_{2}(W_{j}))][{\widehat{A}(TM,\nabla^{TM})}\\
    &\cdot{\rm ch}(\triangle(M)){\rm ch}(\triangle(E),g^{\triangle(E)},d)+2^{2k-1}{\widehat{A}(TM,\nabla^{TM})}\cdot{\rm ch}(\triangle(E),g^{\triangle(E)},d)]\}^{(4k-1)}\\
    &+q^2\{e^{\frac{1}{720}(c_{2}(W_{i})+c_{2}(W_{j}))}\widehat{A}(TM,\nabla^{TM})[{\rm ch}(\triangle(M)\otimes(2\widetilde{T_\mathbf{C}M}+\wedge^2\widetilde{T_\mathbf{C}M}
+\widetilde{T_\mathbf{C}M}\\
&\otimes\widetilde{T_\mathbf{C}M}+S^2\widetilde{T_\mathbf{C}M})){\rm ch}(\triangle(E),g^{\triangle(E)},d)+{\rm ch}(\triangle(M)\otimes2\widetilde{T_{\mathbf{C}}M}){\rm ch}(\triangle(E)\\
&\otimes(2\wedge^2\widetilde{E_\mathbf{C}}-\widetilde{E_\mathbf{C}}\otimes\widetilde{E_\mathbf{C}}+\widetilde{E_\mathbf{C}}),g,d)+{\rm ch}(\triangle(M)){\rm ch}(\triangle(E)\otimes(\wedge^2\widetilde{E_\mathbf{C}}\otimes\wedge^2\widetilde{E_\mathbf{C}}\\
&+2\wedge^4\widetilde{E_\mathbf{C}}-2\widetilde{E_\mathbf{C}}\otimes\wedge^3\widetilde{E_\mathbf{C}}
+2\widetilde{E_\mathbf{C}}\otimes\wedge^2\widetilde{E_\mathbf{C}}-\widetilde{E_\mathbf{C}}\otimes\widetilde{E_\mathbf{C}}\otimes\widetilde{E_\mathbf{C}}
+\widetilde{E_\mathbf{C}}+\wedge^2\widetilde{E_\mathbf{C}})\\
&,g,d)]+2^{2k-1}e^{\frac{1}{720}(c_{2}(W_{i})+c_{2}(W_{j}))}\widehat{A}(TM,\nabla^{TM})[{\rm ch}(\wedge^4\widetilde{T_\mathbf{C}M}+\wedge^2\widetilde{T_\mathbf{C}M}\otimes\widetilde{T_\mathbf{C}M}\\
&+\widetilde{T_\mathbf{C}M}\otimes\widetilde{T_\mathbf{C}M}+S^2\widetilde{T_\mathbf{C}M}+\widetilde{T_\mathbf{C}M}){\rm ch}(\triangle(E),g^{\triangle(E)},d)
+{\rm ch}(\widetilde{T_{\mathbf{C}}M}+\wedge^2\widetilde{T_{\mathbf{C}}M})\\
&\cdot{\rm ch}(\triangle(E)\otimes(2\wedge^2\widetilde{E_\mathbf{C}}-\widetilde
{E_\mathbf{C}}\otimes\widetilde{E_\mathbf{C}}+\widetilde{E_\mathbf{C}}),g,d)
+{\rm ch}(\triangle(E)\otimes(\wedge^2\widetilde{E_\mathbf{C}}\otimes\wedge^2
\widetilde{E_\mathbf{C}}\\
&+2\wedge^4\widetilde{E_\mathbf{C}}-2\widetilde{E_\mathbf{C}}\otimes\wedge^3
\widetilde{E_\mathbf{C}}
+2\widetilde{E_\mathbf{C}}\otimes\wedge^2\widetilde{E_\mathbf{C}}
-\widetilde{E_\mathbf{C}}\otimes\widetilde{E_\mathbf{C}}\otimes\widetilde{E_\mathbf{C}}
+\widetilde{E_\mathbf{C}}+\wedge^2\widetilde{E_\mathbf{C}})\\
&,g,d)]+e^{\frac{1}{720}(c_{2}(W_{i})+c_{2}(W_{j}))}[{\rm ch}(W_i+W_j-16)-\tfrac{1}{30}(c_{2}(W_{i})+c_{2}(W_{j}))]\\
&\cdot[\widehat{A}(TM,\nabla^{TM}){\rm ch}(\triangle(M)\otimes 2\widetilde{T_{\mathbf{C}}M}){\rm ch}(\triangle(E),g^{\triangle(E)},d)+\widehat{A}(TM,\nabla^{TM})\\
&\cdot{\rm ch}(\triangle(M)){\rm ch}(\Delta(E)\otimes(2\wedge^2\widetilde{E_\mathbf{C}}-\widetilde{E_\mathbf{C}}\otimes\widetilde{E_\mathbf{C}}+\widetilde{E_\mathbf{C}}),g,d)+2^{2k-1}\\
&\cdot\widehat{A}(TM,\nabla^{TM}){\rm ch}(\widetilde{T_{\mathbf{C}}M}+\wedge^2\widetilde{T_{\mathbf{C}}M}){\rm ch}(\triangle(E),g^{\triangle(E)},d)+2^{2k-1}\widehat{A}(TM,\nabla^{TM})\end{aligned}
\end{equation}
\begin{equation}
\begin{aligned}
&\cdot{\rm ch}(\Delta(E)\otimes(2\wedge^2\widetilde{E_\mathbf{C}}-\widetilde{E_\mathbf{C}}
    \otimes\widetilde{E_\mathbf{C}}+\widetilde{E_\mathbf{C}}),g,d)]+e^{\frac{1}{720}(c_{2}(W_{i})+c_{2}(W_{j}))}\\
    &\cdot[{\rm ch}(-16W_i-16W_j+W_iW_j+\overline{W_{i}}+\overline{W_{j}}+104)-\tfrac{1}{30}(c_{2}(W_{i})+c_{2}(W_{j}))\\
    &\cdot{\rm ch}(W_i+W_j-16)+\tfrac{1}{2}(-6+\tfrac{1}{30}(c_{2}(W_{i})+c_{2}(W_{j})))\cdot\tfrac{1}{30}(c_{2}(W_{i})+c_{2}(W_{j}))]\\
    &\cdot[{\widehat{A}(TM,\nabla^{TM})}{\rm ch}(\triangle(M)){\rm ch}(\triangle(E),g^{\triangle(E)},d) +2^{2k-1}{\widehat{A}(TM,\nabla^{TM})}\\
    &\cdot{\rm ch}(\triangle(E),g^{\triangle(E)},d)]\}^{(4k-1)}+\cdot\cdot\cdot.\nonumber
\end{aligned}
\end{equation}

When ${\rm dim}M=7$, then $Q(\nabla^{TM},P_{i},g,d,\tau)$ is a modular form over $SL(2,{\bf Z})$ with the weight $12$, it must be a multiple of
\begin{align}
Q(\nabla^{TM},P_{i},P_{j},g,d,\tau)&=\lambda_{1}E_{4}(\tau)^{3}+\lambda_{2}E_{6}(\tau)^{2}\\
&=1+q(720\lambda_{1}-1008\lambda_{2})+q^2(179280\lambda_{1}+220752\lambda_{2})+\cdot\cdot\cdot.\nonumber
\end{align}
where $\lambda_1$, $\lambda_2$ is degree $12$ forms.

From this, we obtain
\begin{equation}
\begin{aligned}
\lambda_1
=&\tfrac{1008}{1728}\{ e^{\frac{1}{720}(c_{2}(W_{i})+c_{2}(W_{j}))}[{\widehat{A}(TM,\nabla^{TM})}{\rm ch}(\triangle(M)){\rm ch}(\triangle(E),g^{\triangle(E)},d) \\
&+8{\widehat{A}(TM,\nabla^{TM})}{\rm ch}(\triangle(E),g^{\triangle(E)},d)]\}^{(7)}+\tfrac{1}{1728}\{ e^{\frac{1}{720}(c_{2}(W_{i})+c_{2}(W_{j}))}\\
&\cdot{\widehat{A}(TM,\nabla^{TM})}[{\rm ch}(\triangle(M)\otimes2\widetilde{T_{\mathbf{C}}M}){\rm ch}(\triangle(E),g^{\triangle(E)},d)+{\rm ch}(\triangle(M))\\
&\cdot{\rm ch}(\Delta(E)\otimes(2\wedge^2\widetilde{E_\mathbf{C}}-\widetilde{E_\mathbf{C}}
    \otimes\widetilde{E_\mathbf{C}}+\widetilde{E_\mathbf{C}}),g,d)]+8e^{\frac{1}{720}(c_{2}(W_{i})+c_{2}(W_{j}))}\\
    &\cdot\widehat{A}(TM,\nabla^{TM})[{\rm ch}(\widetilde{T_{\mathbf{C}}M}+\wedge^2\widetilde{T_{\mathbf{C}}M}){\rm ch}(\triangle(E),g^{\triangle(E)},d)+{\rm ch}(\Delta(E)\\
    &\otimes(2\wedge^2\widetilde{E_\mathbf{C}}-\widetilde{E_\mathbf{C}}\otimes\widetilde{E_\mathbf{C}}+\widetilde{E_\mathbf{C}}),g,d)]+e^{\frac{1}{720}(c_{2}(W_{i})+c_{2}(W_{j}))}[{\rm ch}(W_i+W_j\\
    &-16)-\tfrac{1}{30}(c_{2}(W_{i})+c_{2}(W_{j}))][{\widehat{A}(TM,\nabla^{TM})}{\rm ch}(\triangle(M)){\rm ch}(\triangle(E),g^{\triangle(E)},d)\\
    &+8{\widehat{A}(TM,\nabla^{TM})}\cdot{\rm ch}(\triangle(E),g^{\triangle(E)},d)]\}^{(7)},
    \end{aligned}
  \end{equation}
  and
  \begin{equation}
\begin{aligned}
\lambda_2
=&\tfrac{720}{1728}\{ e^{\frac{1}{720}(c_{2}(W_{i})+c_{2}(W_{j}))}[{\widehat{A}(TM,\nabla^{TM})}{\rm ch}(\triangle(M)){\rm ch}(\triangle(E),g^{\triangle(E)},d) \\
&+8{\widehat{A}(TM,\nabla^{TM})}{\rm ch}(\triangle(E),g^{\triangle(E)},d)]\}^{(7)}-\tfrac{1}{1728}\{ e^{\frac{1}{720}(c_{2}(W_{i})+c_{2}(W_{j}))}\\
&\cdot{\widehat{A}(TM,\nabla^{TM})}[{\rm ch}(\triangle(M)\otimes2\widetilde{T_{\mathbf{C}}M}){\rm ch}(\triangle(E),g^{\triangle(E)},d)+{\rm ch}(\triangle(M))\\
&\cdot{\rm ch}(\Delta(E)\otimes(2\wedge^2\widetilde{E_\mathbf{C}}-\widetilde{E_\mathbf{C}}
    \otimes\widetilde{E_\mathbf{C}}+\widetilde{E_\mathbf{C}}),g,d)]+8e^{\frac{1}{720}(c_{2}(W_{i})+c_{2}(W_{j}))}\\
    &\cdot\widehat{A}(TM,\nabla^{TM})[{\rm ch}(\widetilde{T_{\mathbf{C}}M}+\wedge^2\widetilde{T_{\mathbf{C}}M}){\rm ch}(\triangle(E),g^{\triangle(E)},d)+{\rm ch}(\Delta(E)\\
    &\otimes(2\wedge^2\widetilde{E_\mathbf{C}}-\widetilde{E_\mathbf{C}}\otimes\widetilde{E_\mathbf{C}}+\widetilde{E_\mathbf{C}}),g,d)]+e^{\frac{1}{720}(c_{2}(W_{i})+c_{2}(W_{j}))}[{\rm ch}(W_i+W_j\\
    &-16)-\tfrac{1}{30}(c_{2}(W_{i})+c_{2}(W_{j}))][{\widehat{A}(TM,\nabla^{TM})}{\rm ch}(\triangle(M)){\rm ch}(\triangle(E),g^{\triangle(E)},d)\\
    &+8{\widehat{A}(TM,\nabla^{TM})}\cdot{\rm ch}(\triangle(E),g^{\triangle(E)},d)]\}^{(7)}.
    \end{aligned}
  \end{equation}

Theorem 4.2 follows from comparing the coefficients of the $q^2$ terms.
\end{proof}

\begin{thm}
For $11$-dimensional spin manifold, when $c_3(E,g,d)=0,$ we conclude that
  \begin{equation}
   \begin{aligned}
    &\{ e^{\frac{1}{720}(c_{2}(W_{i})+c_{2}(W_{j}))}{\widehat{A}(TM,\nabla^{TM})}[{\rm ch}(\triangle(M)\otimes2\widetilde{T_{\mathbf{C}}M}){\rm ch}(\triangle(E),g^{\triangle(E)},d)\\
    &+{\rm ch}(\triangle(M)){\rm ch}(\Delta(E)\otimes(2\wedge^2\widetilde{E_\mathbf{C}}-\widetilde{E_\mathbf{C}}
    \otimes\widetilde{E_\mathbf{C}}+\widetilde{E_\mathbf{C}}),g,d)]+32e^{\frac{1}{720}c_{2}(W_{i})}\\
    &\cdot e^{\frac{1}{720}c_{2}(W_{j})}\widehat{A}(TM,\nabla^{TM})[{\rm ch}(\widetilde{T_{\mathbf{C}}M}+\wedge^2\widetilde{T_{\mathbf{C}}M}){\rm ch}(\triangle(E),g^{\triangle(E)},d)+{\rm ch}(\Delta(E)\\
    &\otimes(2\wedge^2\widetilde{E_\mathbf{C}}-\widetilde{E_\mathbf{C}}\otimes\widetilde{E_\mathbf{C}}+\widetilde{E_\mathbf{C}}),g,d)]+e^{\frac{1}{720}(c_{2}(W_{i})+c_{2}(W_{j}))}[{\rm ch}(W_i+W_j-16)\\
    &-\tfrac{1}{30}(c_{2}(W_{i})+c_{2}(W_{j}))][{\widehat{A}(TM,\nabla^{TM})}{\rm ch}(\triangle(M)){\rm ch}(\triangle(E),g^{\triangle(E)},d)+32\\
    &\cdot{\widehat{A}(TM,\nabla^{TM})}{\rm ch}(\triangle(E),g^{\triangle(E)},d)]\}^{(11)}=-24 \{ e^{\frac{1}{720}(c_{2}(W_{i})+c_{2}(W_{j}))}[{\widehat{A}(TM,\nabla^{TM})}\\
    &\cdot{\rm ch}(\triangle(M)){\rm ch}(\triangle(E),g^{\triangle(E)},d) +32{\widehat{A}(TM,\nabla^{TM})}{\rm ch}(\triangle(E),g^{\triangle(E)},d)]\}^{(11)},
   \end{aligned}
  \end{equation}
  and
  \begin{equation}
   \begin{aligned}
    &\{e^{\frac{1}{720}(c_{2}(W_{i})+c_{2}(W_{j}))}\widehat{A}(TM,\nabla^{TM})[{\rm ch}(\triangle(M)\otimes(2\widetilde{T_\mathbf{C}M}+\wedge^2\widetilde{T_\mathbf{C}M}
+\widetilde{T_\mathbf{C}M}\\
&\otimes\widetilde{T_\mathbf{C}M}+S^2\widetilde{T_\mathbf{C}M})){\rm ch}(\triangle(E),g^{\triangle(E)},d)+{\rm ch}(\triangle(M)\otimes2\widetilde{T_{\mathbf{C}}M}){\rm ch}(\triangle(E)\\
&\otimes(2\wedge^2\widetilde{E_\mathbf{C}}-\widetilde{E_\mathbf{C}}\otimes\widetilde{E_\mathbf{C}}+\widetilde{E_\mathbf{C}}),g,d)+{\rm ch}(\triangle(M)){\rm ch}(\triangle(E)\otimes(\wedge^2\widetilde{E_\mathbf{C}}\\
&\otimes\wedge^2\widetilde{E_\mathbf{C}}
+2\wedge^4\widetilde{E_\mathbf{C}}-2\widetilde{E_\mathbf{C}}\otimes\wedge^3\widetilde{E_\mathbf{C}}
+2\widetilde{E_\mathbf{C}}\otimes\wedge^2\widetilde{E_\mathbf{C}}-\widetilde{E_\mathbf{C}}\otimes\widetilde{E_\mathbf{C}}\otimes\widetilde{E_\mathbf{C}}\\
&+\widetilde{E_\mathbf{C}}+\wedge^2\widetilde{E_\mathbf{C}}),g,d)]+32e^{\frac{1}{720}(c_{2}(W_{i})+c_{2}(W_{j}))}\widehat{A}(TM,\nabla^{TM})[{\rm ch}(\wedge^4\widetilde{T_\mathbf{C}M}\\
&+\wedge^2\widetilde{T_\mathbf{C}M}\otimes\widetilde{T_\mathbf{C}M}+\widetilde{T_\mathbf{C}M}\otimes\widetilde{T_\mathbf{C}M}+S^2\widetilde{T_\mathbf{C}M}+\widetilde{T_\mathbf{C}M}){\rm ch}(\triangle(E),g^{\triangle(E)},d)\\
&+{\rm ch}(\widetilde{T_{\mathbf{C}}M}+\wedge^2\widetilde{T_{\mathbf{C}}M}){\rm ch}(\triangle(E)\otimes(2\wedge^2\widetilde{E_\mathbf{C}}-\widetilde
{E_\mathbf{C}}\otimes\widetilde{E_\mathbf{C}}+\widetilde{E_\mathbf{C}}),g,d)\\
&+{\rm ch}(\triangle(E)\otimes(\wedge^2\widetilde{E_\mathbf{C}}\otimes\wedge^2
\widetilde{E_\mathbf{C}}+2\wedge^4\widetilde{E_\mathbf{C}}-2\widetilde{E_\mathbf{C}}\otimes\wedge^3
\widetilde{E_\mathbf{C}}
+2\widetilde{E_\mathbf{C}}\otimes\wedge^2\widetilde{E_\mathbf{C}}\\
&-\widetilde{E_\mathbf{C}}\otimes\widetilde{E_\mathbf{C}}\otimes\widetilde{E_\mathbf{C}}
+\widetilde{E_\mathbf{C}}+\wedge^2\widetilde{E_\mathbf{C}}),g,d)]+e^{\frac{1}{720}(c_{2}(W_{i})+c_{2}(W_{j}))}[{\rm ch}(W_i+W_j\\
&-16)-\tfrac{1}{30}(c_{2}(W_{i})+c_{2}(W_{j}))][\widehat{A}(TM,\nabla^{TM}){\rm ch}(\triangle(M)\otimes 2\widetilde{T_{\mathbf{C}}M})\\
&\cdot{\rm ch}(\triangle(E),g^{\triangle(E)},d)+\widehat{A}(TM,\nabla^{TM}){\rm ch}(\triangle(M)){\rm ch}(\Delta(E)\otimes(2\wedge^2\widetilde{E_\mathbf{C}}\\
&-\widetilde{E_\mathbf{C}}\otimes\widetilde{E_\mathbf{C}}+\widetilde{E_\mathbf{C}}),g,d)+32\widehat{A}(TM,\nabla^{TM}){\rm ch}(\widetilde{T_{\mathbf{C}}M}+\wedge^2\widetilde{T_{\mathbf{C}}M})\\
&\cdot{\rm ch}(\triangle(E),g^{\triangle(E)},d)+32\widehat{A}(TM,\nabla^{TM}){\rm ch}(\Delta(E)\otimes(2\wedge^2\widetilde{E_\mathbf{C}}-\widetilde{E_\mathbf{C}}\\
   & \otimes\widetilde{E_\mathbf{C}}+\widetilde{E_\mathbf{C}}),g,d)]+e^{\frac{1}{720}(c_{2}(W_{i})+c_{2}(W_{j}))}\cdot[{\rm ch}(-16W_i-16W_j+W_iW_j\\
   &+\overline{W_{i}}+\overline{W_{j}}+104)-\tfrac{1}{30}(c_{2}(W_{i})+c_{2}(W_{j})){\rm ch}(W_i+W_j-16)+\tfrac{1}{2}(-6\\
   &+\tfrac{1}{30}(c_{2}(W_{i})+c_{2}(W_{j})))\cdot\tfrac{1}{30}(c_{2}(W_{i})+c_{2}(W_{j}))][{\widehat{A}(TM,\nabla^{TM})}{\rm ch}(\triangle(M))\\
   &\cdot{\rm ch}(\triangle(E),g^{\triangle(E)},d) +32{\widehat{A}(TM,\nabla^{TM})}{\rm ch}(\triangle(E),g^{\triangle(E)},d)]\}^{(11)}\\
   &=-196632 \{ e^{\frac{1}{720}(c_{2}(W_{i})+c_{2}(W_{j}))}[{\widehat{A}(TM,\nabla^{TM})}{\rm ch}(\triangle(M)){\rm ch}(\triangle(E),g^{\triangle(E)},d) \\ &+32{\widehat{A}(TM,\nabla^{TM})}{\rm ch}(\triangle(E),g^{\triangle(E)},d)]\}^{(11)}.
   \end{aligned}
  \end{equation}
\end{thm}

\begin{proof}
Similarly, we have
\begin{align}
E_{4}(\tau)^{2}E_{6}(\tau)=1-24q-196632q^{2}+\cdot\cdot\cdot.
\end{align}
Since $Q(\nabla^{TM},P_{i},P_{j},g,d,\tau)$ is a modular form over $SL(2,{\bf Z})$ with the weight $14$,  it must be a multiple of $E_{4}(\tau)^{2}E_{6}(\tau).$
This leads us to the following conclusion.
\end{proof}

\begin{thm}
For $15$-dimensional spin manifold, when $c_3(E,g,d)=0,$ we conclude that
   \begin{equation}
    \begin{aligned}
&\{e^{\frac{1}{720}(c_{2}(W_{i})+c_{2}(W_{j}))}\widehat{A}(TM,\nabla^{TM})[{\rm ch}(\triangle(M)\otimes(2\widetilde{T_\mathbf{C}M}+\wedge^2\widetilde{T_\mathbf{C}M}
+\widetilde{T_\mathbf{C}M}\\
&\otimes\widetilde{T_\mathbf{C}M}+S^2\widetilde{T_\mathbf{C}M})){\rm ch}(\triangle(E),g^{\triangle(E)},d)+{\rm ch}(\triangle(M)\otimes2\widetilde{T_{\mathbf{C}}M}){\rm ch}(\triangle(E)\\
&\otimes(2\wedge^2\widetilde{E_\mathbf{C}}-\widetilde{E_\mathbf{C}}\otimes\widetilde{E_\mathbf{C}}+\widetilde{E_\mathbf{C}}),g,d)+{\rm ch}(\triangle(M)){\rm ch}(\triangle(E)\otimes(\wedge^2\widetilde{E_\mathbf{C}}\\
&\otimes\wedge^2\widetilde{E_\mathbf{C}}+2\wedge^4\widetilde{E_\mathbf{C}}-2\widetilde{E_\mathbf{C}}\otimes\wedge^3\widetilde{E_\mathbf{C}}
+2\widetilde{E_\mathbf{C}}\otimes\wedge^2\widetilde{E_\mathbf{C}}-\widetilde{E_\mathbf{C}}\otimes\widetilde{E_\mathbf{C}}\otimes\widetilde{E_\mathbf{C}}\\
&+\widetilde{E_\mathbf{C}}+\wedge^2\widetilde{E_\mathbf{C}}),g,d)]+128e^{\frac{1}{720}(c_{2}(W_{i})+c_{2}(W_{j}))}\widehat{A}(TM,\nabla^{TM})[{\rm ch}(\wedge^4\widetilde{T_\mathbf{C}M}\\
&+\wedge^2\widetilde{T_\mathbf{C}M}\otimes\widetilde{T_\mathbf{C}M}+\widetilde{T_\mathbf{C}M}\otimes\widetilde{T_\mathbf{C}M}+S^2\widetilde{T_\mathbf{C}M}+\widetilde{T_\mathbf{C}M}){\rm ch}(\triangle(E),g^{\triangle(E)},d)\\
&+{\rm ch}(\widetilde{T_{\mathbf{C}}M}+\wedge^2\widetilde{T_{\mathbf{C}}M}){\rm ch}(\triangle(E)\otimes(2\wedge^2\widetilde{E_\mathbf{C}}-\widetilde
{E_\mathbf{C}}\otimes\widetilde{E_\mathbf{C}}+\widetilde{E_\mathbf{C}}),g,d)\\
&+{\rm ch}(\triangle(E)\otimes(\wedge^2\widetilde{E_\mathbf{C}}\otimes\wedge^2
\widetilde{E_\mathbf{C}}
+2\wedge^4\widetilde{E_\mathbf{C}}-2\widetilde{E_\mathbf{C}}\otimes\wedge^3
\widetilde{E_\mathbf{C}}
+2\widetilde{E_\mathbf{C}}\otimes\wedge^2\widetilde{E_\mathbf{C}}\\
&-\widetilde{E_\mathbf{C}}\otimes\widetilde{E_\mathbf{C}}\otimes\widetilde{E_\mathbf{C}}
+\widetilde{E_\mathbf{C}}+\wedge^2\widetilde{E_\mathbf{C}}),g,d)]+e^{\frac{1}{720}(c_{2}(W_{i})+c_{2}(W_{j}))}[{\rm ch}(W_i+W_j-16)\\
&-\tfrac{1}{30}(c_{2}(W_{i})+c_{2}(W_{j}))][\widehat{A}(TM,\nabla^{TM}){\rm ch}(\triangle(M)\otimes 2\widetilde{T_{\mathbf{C}}M}){\rm ch}(\triangle(E),g^{\triangle(E)},d)\\
&+\widehat{A}(TM,\nabla^{TM}){\rm ch}(\triangle(M)){\rm ch}(\Delta(E)\otimes(2\wedge^2\widetilde{E_\mathbf{C}}-\widetilde{E_\mathbf{C}}\otimes\widetilde{E_\mathbf{C}}+\widetilde{E_\mathbf{C}}),g,d)\\
&+128\widehat{A}(TM,\nabla^{TM}){\rm ch}(\widetilde{T_{\mathbf{C}}M}+\wedge^2\widetilde{T_{\mathbf{C}}M}){\rm ch}(\triangle(E),g^{\triangle(E)},d)+128\widehat{A}(TM,\nabla^{TM})\\
&\cdot{\rm ch}(\Delta(E)\otimes(2\wedge^2\widetilde{E_\mathbf{C}}-\widetilde{E_\mathbf{C}}
    \otimes\widetilde{E_\mathbf{C}}+\widetilde{E_\mathbf{C}}),g,d)]+e^{\frac{1}{720}(c_{2}(W_{i})+c_{2}(W_{j}))}[{\rm ch}(-16W_i\\
    &-16W_j+W_iW_j+\overline{W_{i}}+\overline{W_{j}}+104)-\tfrac{1}{30}(c_{2}(W_{i})+c_{2}(W_{j})){\rm ch}(W_i+W_j-16)\\
    &+\tfrac{1}{2}(-6+\tfrac{1}{30}(c_{2}(W_{i})+c_{2}(W_{j})))\cdot\tfrac{1}{30}(c_{2}(W_{i})+c_{2}(W_{j}))][{\widehat{A}(TM,\nabla^{TM})}{\rm ch}(\triangle(M))\\
    &\cdot{\rm ch}(\triangle(E),g^{\triangle(E)},d) +128{\widehat{A}(TM,\nabla^{TM})}{\rm ch}(\triangle(E),g^{\triangle(E)},d)]\}^{(15)}=146880 \\
    &\cdot \{e^{\frac{1}{720}(c_{2}(W_{i})+c_{2}(W_{j}))}[{\widehat{A}(TM,\nabla^{TM})}{\rm ch}(\triangle(M)){\rm ch}(\triangle(E),g^{\triangle(E)},d) +128\\
    &\cdot{\widehat{A}(TM,\nabla^{TM})}{\rm ch}(\triangle(E),g^{\triangle(E)},d)]\}^{(15)}+194\{ e^{\frac{1}{720}(c_{2}(W_{i})+c_{2}(W_{j}))}{\widehat{A}(TM,\nabla^{TM})}\\
    &\cdot[{\rm ch}(\triangle(M)\otimes 2\widetilde{T_{\mathbf{C}}M}){\rm ch}(\triangle(E),g^{\triangle(E)},d)+{\rm ch}(\triangle(M)){\rm ch}(\Delta(E)\otimes(2\wedge^2\widetilde{E_\mathbf{C}}\\
    &-\widetilde{E_\mathbf{C}}
    \otimes\widetilde{E_\mathbf{C}}+\widetilde{E_\mathbf{C}}),g,d)]+128e^{\frac{1}{720}(c_{2}(W_{i})+c_{2}(W_{j}))}\widehat{A}(TM,\nabla^{TM})[{\rm ch}(\widetilde{T_{\mathbf{C}}M}\\
    &+\wedge^2\widetilde{T_{\mathbf{C}}M}){\rm ch}(\triangle(E),g^{\triangle(E)},d)+{\rm ch}(\Delta(E)\otimes(2\wedge^2\widetilde{E_\mathbf{C}}-\widetilde{E_\mathbf{C}} \otimes\widetilde{E_\mathbf{C}}+\widetilde{E_\mathbf{C}}),g,d)]\\
    &+e^{\frac{1}{720}(c_{2}(W_{i})+c_{2}(W_{j}))}[{\rm ch}(W_i+W_j-16)-\tfrac{1}{30}(c_{2}(W_{i})+c_{2}(W_{j}))][{\widehat{A}(TM,\nabla^{TM})}\\
    &\cdot{\rm ch}(\triangle(M)){\rm ch}(\triangle(E),g^{\triangle(E)},d)+128{\widehat{A}(TM,\nabla^{TM})}{\rm ch}(\triangle(E),g^{\triangle(E)},d)]\}^{(15)}.
    \end{aligned}
  \end{equation}
\end{thm}

\begin{proof}
It is easy to obtain $Q(\nabla^{TM},P_{i},P_{j},g,d,\tau)$ is a modular form over $SL(2,{\bf Z})$ with the weight $16.$
Consequently,
\begin{align}
&Q(\nabla^{TM},P_{i},P_{j},g,d,\tau)\nonumber\\
=&\lambda_1E_{4}(\tau)E_{6}(\tau)^{2}+\lambda_2E_{4}(\tau)^{4}E_{6}\\
=&\lambda_1+\lambda_2+q(-768\lambda_1+960\lambda_2)+q^{2}(-19008\lambda_1+354240\lambda_2)+\cdot\cdot\cdot.\nonumber
\end{align}

To find $\lambda_1$, $\lambda_2$ and Theorem 4.4, we compare the corresponding coefficients.
\end{proof}

\subsection{Modular forms and Witten genus in $(4k-1)$-dimensional spin$^c$ manifolds}\hfill\\

Set $M$ be a closed oriented ${\rm spin}^{c}$ manifold of dimension $4k-1$.
To study $E_{8}\times E_{8}$ bundles, we extend the definition of $\widetilde{Q}(\nabla^{TM},\nabla^{L},P_{i},g,d,\tau)$, which is originally associated with $E_{8}$ bundles.
\begin{equation}
\begin{aligned}
\widetilde{Q}(\nabla^{TM},\nabla^{L},P_{i},P_{j},g,d,\tau)=&\{e^{\frac{1}{24}E_{2}(\tau)\cdot\frac{1}{30}(c_{2}(W_{i})+c_{2}(W_{j}))} \widehat{A}(TM,\nabla^{TM}){\rm exp}(\frac{c}{2})\\
&\cdot{\rm ch}(\Theta(T_{\mathbf{C}}M,L_{\bf{R}}\otimes\mathbf{C})){\rm ch}(Q(E),g^{Q(E)},d,\tau)\\
&\cdot\varphi(\tau)^{(16)}ch(\mathcal{V}_{i})ch(\mathcal{V}_{j})\}^{(4k-1)}.
\end{aligned}
\end{equation}

\begin{thm}
Let ${\rm dim}M=4k-1,$ for $k\leq 4.$  If $3p_1(L)-p_1(M)=0$ and $c_3(E,g,d)=0,$ then $\widetilde{Q}(\nabla^{TM},\nabla^{L},P_{i},P_{j},g,d,\tau)$ is a modular form over $SL(2,{\bf Z})$ with the weight $2k+8.$
\end{thm}

A straightforward proof gives
\begin{thm}
For $7$-dimensional spin manifold, when $3p_1(L)-p_1(M)=0$ and $c_3(E,g,d)=0,$ we conclude that
\begin{equation}
\begin{aligned}
&\{ {\widehat{A}(TM,\nabla^{TM})}{\rm exp}(\frac{c}{2}){\rm ch}(\triangle(E),g^{\triangle(E)},d)[e^{\frac{1}{720}(c_{2}(W_{i})+c_{2}(W_{j}))}\cdot\tfrac{1}{2}(-6+\tfrac{1}{30}(c_{2}(W_{i})\\
&+c_{2}(W_{j})))\cdot\tfrac{1}{30}(c_{2}(W_{i})+c_{2}(W_{j}))-e^{\frac{1}{720}(c_{2}(W_{i})+c_{2}(W_{j}))}\cdot\tfrac{1}{30}(c_{2}(W_{i})+c_{2}(W_{j}))\\
&\cdot{\rm ch}(W_i+W_j-16)+e^{\frac{1}{720}(c_{2}(W_{i})+c_{2}(W_{j}))}{\rm ch}(-16W_i-16W_j+W_iW_j+\overline{W_{i}}\\
&+\overline{W_{j}}+104)-e^{\frac{1}{720}(c_{2}(W_{i})+c_{2}(W_{j}))}\cdot\tfrac{1}{30}(c_{2}(W_{i})+c_{2}(W_{j})){\rm ch}(\widetilde{T_\mathbf{C}M}+\widetilde{E_\mathbf{C}}-\widetilde{E_\mathbf{C}}\\
    &\otimes\widetilde{E_\mathbf{C}}+2\wedge^2\widetilde{E_\mathbf{C}}+\widetilde{L_{\bf{R}}\otimes\mathbf{C}}-(\widetilde{L_{\bf{R}}\otimes\mathbf{C}})\otimes(\widetilde{L_{\bf{R}}\otimes\mathbf{C}})+2\wedge^2\widetilde{L_{\bf{R}}\otimes\mathbf{C}})\\
    &+e^{\frac{1}{720}(c_{2}(W_{i})+c_{2}(W_{j}))}{\rm ch}(W_i+W_j-16){\rm ch}(\widetilde{T_\mathbf{C}M}+\widetilde{E_\mathbf{C}}-\widetilde{E_\mathbf{C}}
    \otimes\widetilde{E_\mathbf{C}}+2\wedge^2\widetilde{E_\mathbf{C}}\\
    &+\widetilde{L_{\bf{R}}\otimes\mathbf{C}}-(\widetilde{L_{\bf{R}}\otimes\mathbf{C}})
    \otimes(\widetilde{L_{\bf{R}}\otimes\mathbf{C}})+2\wedge^2\widetilde{L_{\bf{R}}\otimes\mathbf{C}})+e^{\frac{1}{720}(c_{2}(W_{i})+c_{2}(W_{j}))}\\
    &\cdot{\rm ch}(\widetilde{T_\mathbf{C}M}+S^2\widetilde{T_\mathbf{C}M}+\widetilde{T_\mathbf{C}M}\otimes(\widetilde{E_\mathbf{C}}-\widetilde{E_\mathbf{C}}
    \otimes\widetilde{E_\mathbf{C}}+2\wedge^2\widetilde{E_\mathbf{C}}+\widetilde{L_{\bf{R}}\otimes\mathbf{C}}\\
    &-(\widetilde{L_{\bf{R}}\otimes\mathbf{C}})
    \otimes(\widetilde{L_{\bf{R}}\otimes\mathbf{C}})+2\wedge^2\widetilde{L_{\bf{R}}\otimes\mathbf{C}})+\wedge^2\widetilde{E_\mathbf{C}}\otimes\wedge^2\widetilde{E_\mathbf{C}}
+2\wedge^4\widetilde{E_\mathbf{C}}-2\widetilde{E_\mathbf{C}}\end{aligned}
\end{equation}
\begin{equation}
\begin{aligned}
&\otimes\wedge^3\widetilde{E_\mathbf{C}}
+2\widetilde{E_\mathbf{C}}\otimes\wedge^2\widetilde{E_\mathbf{C}}-\widetilde{E_\mathbf{C}}\otimes
\widetilde{E_\mathbf{C}}\otimes\widetilde{E_\mathbf{C}}+\widetilde{E_\mathbf{C}}
+\wedge^2\widetilde{E_\mathbf{C}}+(2\wedge^2\widetilde{E_\mathbf{C}}-\widetilde{E_\mathbf{C}}\\
&\otimes\widetilde{E_\mathbf{C}}+\widetilde{E_\mathbf{C}})\otimes(2\wedge^2\widetilde{L_{\bf{R}}\otimes\mathbf{C}}-(\widetilde{L_{\bf{R}}\otimes\mathbf{C}})\otimes(\widetilde{L_{\bf{R}}\otimes\mathbf{C}})+\widetilde{L_{\bf{R}}\otimes\mathbf{C}})+\wedge^2\widetilde{L_{\bf{R}}\otimes\mathbf{C}}\\
&\otimes\wedge^2\widetilde{L_{\bf{R}}
\otimes\mathbf{C}}+2\wedge^4\widetilde{L_{\bf{R}}\otimes\mathbf{C}}-2\widetilde{L_{\bf{R}}\otimes\mathbf{C}}\otimes\wedge^3\widetilde{L_{\bf{R}}
\otimes\mathbf{C}}+2\widetilde{L_{\bf{R}}\otimes\mathbf{C}}\otimes\wedge^2\widetilde{L_{\bf{R}}\otimes\mathbf{C}}\\
&-\widetilde{L_{\bf{R}}\otimes\mathbf{C}}\otimes\widetilde{L_{\bf{R}}
\otimes\mathbf{C}}\otimes\widetilde{L_{\bf{R}}
\otimes\mathbf{C}}+\widetilde{L_{\bf{R}}
\otimes\mathbf{C}}+\wedge^2\widetilde{L_{\bf{R}}\otimes\mathbf{C}})]\}^{(7)}=196560\\
    &\cdot\{ e^{\frac{1}{720}(c_{2}(W_{i})+c_{2}(W_{j}))}{\widehat{A}(TM,\nabla^{TM})}{\rm exp}(\frac{c}{2}){\rm ch}(\triangle(E),g^{\triangle(E)},d) \}^{(7)}-24\{ {\rm exp}(\frac{c}{2})\\
    &\cdot{\widehat{A}(TM,\nabla^{TM})}{\rm ch}(\triangle(E),g^{\triangle(E)},d)[-e^{\frac{1}{720}(c_{2}(W_{i})+c_{2}(W_{j}))}\cdot\tfrac{1}{30}(c_{2}(W_{i})+c_{2}(W_{j}))\\
    &+e^{\frac{1}{720}(c_{2}(W_{i})+c_{2}(W_{j}))}{\rm ch}(W_i+W_j-16)+e^{\frac{1}{720}(c_{2}(W_{i})+c_{2}(W_{j}))}{\rm ch}(\widetilde{T_\mathbf{C}M}+\widetilde{E_\mathbf{C}}\\
    &-\widetilde{E_\mathbf{C}}
    \otimes\widetilde{E_\mathbf{C}}+2\wedge^2\widetilde{E_\mathbf{C}}+\widetilde{L_{\bf{R}}\otimes\mathbf{C}}-(\widetilde{L_{\bf{R}}\otimes\mathbf{C}})
    \otimes(\widetilde{L_{\bf{R}}\otimes\mathbf{C}})+2\wedge^2\widetilde{L_{\bf{R}}\otimes\mathbf{C}})]\}^{(7)}.\nonumber
\end{aligned}
\end{equation}
\end{thm}

\begin{thm}
For $11$-dimensional spin manifold, when $3p_1(L)-p_1(M)=0$ and $c_3(E,g,d)=0,$ we conclude that
 \begin{equation}
\begin{aligned}
& \{ {\widehat{A}(TM,\nabla^{TM})}{\rm exp}(\frac{c}{2}){\rm ch}(\triangle(E),g^{\triangle(E)},d)[-e^{\frac{1}{720}(c_{2}(W_{i})+c_{2}(W_{j}))}\cdot\tfrac{1}{30}(c_{2}(W_{i})+c_{2}(W_{j}))\\
&+e^{\frac{1}{720}(c_{2}(W_{i})+c_{2}(W_{j}))}{\rm ch}(W_i+W_j-16)+e^{\frac{1}{720}(c_{2}(W_{i})+c_{2}(W_{j}))}{\rm ch}(\widetilde{T_\mathbf{C}M}+\widetilde{E_\mathbf{C}}\\
&-\widetilde{E_\mathbf{C}}
    \otimes\widetilde{E_\mathbf{C}}+2\wedge^2\widetilde{E_\mathbf{C}}+\widetilde{L_{\bf{R}}\otimes\mathbf{C}}-(\widetilde{L_{\bf{R}}\otimes\mathbf{C}})
    \otimes(\widetilde{L_{\bf{R}}\otimes\mathbf{C}})+2\wedge^2\widetilde{L_{\bf{R}}\otimes\mathbf{C}})]\}^{(11)}\\
    &=-24\{ e^{\frac{1}{720}(c_{2}(W_{i})+c_{2}(W_{j}))}{\widehat{A}(TM,\nabla^{TM})}{\rm exp}(\frac{c}{2}){\rm ch}(\triangle(E),g^{\triangle(E)},d) \}^{(11)},
\end{aligned}
\end{equation}
  and
 \begin{equation}
\begin{aligned}
&\{ {\widehat{A}(TM,\nabla^{TM})}{\rm exp}(\frac{c}{2}){\rm ch}(\triangle(E),g^{\triangle(E)},d)[e^{\frac{1}{720}(c_{2}(W_{i})+c_{2}(W_{j}))}\cdot\tfrac{1}{2}(-6+\tfrac{1}{30}(c_{2}(W_{i})\\
&+c_{2}(W_{j})))\cdot\tfrac{1}{30}(c_{2}(W_{i})+c_{2}(W_{j}))-e^{\frac{1}{720}(c_{2}(W_{i})+c_{2}(W_{j}))}\cdot\tfrac{1}{30}(c_{2}(W_{i})+c_{2}(W_{j}))\\
&\cdot{\rm ch}(W_i+W_j-16)+e^{\frac{1}{720}(c_{2}(W_{i})+c_{2}(W_{j}))}{\rm ch}(-16W_i-16W_j+W_iW_j+\overline{W_{i}}\\
&+\overline{W_{j}}+104)-e^{\frac{1}{720}(c_{2}(W_{i})+c_{2}(W_{j}))}\cdot\tfrac{1}{30}(c_{2}(W_{i})+c_{2}(W_{j})){\rm ch}(\widetilde{T_\mathbf{C}M}+\widetilde{E_\mathbf{C}}\\
&-\widetilde{E_\mathbf{C}}
    \otimes\widetilde{E_\mathbf{C}}+2\wedge^2\widetilde{E_\mathbf{C}}+\widetilde{L_{\bf{R}}\otimes\mathbf{C}}-(\widetilde{L_{\bf{R}}\otimes\mathbf{C}})\otimes(\widetilde{L_{\bf{R}}\otimes\mathbf{C}})+2\wedge^2\widetilde{L_{\bf{R}}\otimes\mathbf{C}})\\
    &+e^{\frac{1}{720}(c_{2}(W_{i})+c_{2}(W_{j}))}{\rm ch}(W_i+W_j-16){\rm ch}(\widetilde{T_\mathbf{C}M}+\widetilde{E_\mathbf{C}}-\widetilde{E_\mathbf{C}}
    \otimes\widetilde{E_\mathbf{C}}+2\wedge^2\widetilde{E_\mathbf{C}}\\
    &+\widetilde{L_{\bf{R}}\otimes\mathbf{C}}-(\widetilde{L_{\bf{R}}\otimes\mathbf{C}})
    \otimes(\widetilde{L_{\bf{R}}\otimes\mathbf{C}})+2\wedge^2\widetilde{L_{\bf{R}}\otimes\mathbf{C}})+e^{\frac{1}{720}(c_{2}(W_{i})+c_{2}(W_{j}))}{\rm ch}(\widetilde{T_\mathbf{C}M}\\
    &+S^2\widetilde{T_\mathbf{C}M}+\widetilde{T_\mathbf{C}M}\otimes(\widetilde{E_\mathbf{C}}-\widetilde{E_\mathbf{C}}
    \otimes\widetilde{E_\mathbf{C}}+2\wedge^2\widetilde{E_\mathbf{C}}+\widetilde{L_{\bf{R}}\otimes\mathbf{C}}-(\widetilde{L_{\bf{R}}\otimes\mathbf{C}})\\
   &\otimes(\widetilde{L_{\bf{R}}\otimes\mathbf{C}})+2\wedge^2\widetilde{L_{\bf{R}}\otimes\mathbf{C}})+\wedge^2\widetilde{E_\mathbf{C}}\otimes\wedge^2\widetilde{E_\mathbf{C}}
+2\wedge^4\widetilde{E_\mathbf{C}}-2\widetilde{E_\mathbf{C}}\otimes\wedge^3\widetilde{E_\mathbf{C}}\end{aligned}
\end{equation}
\begin{equation}
\begin{aligned}
&+2\widetilde{E_\mathbf{C}}\otimes\wedge^2\widetilde{E_\mathbf{C}}-\widetilde{E_\mathbf{C}}\otimes
\widetilde{E_\mathbf{C}}\otimes\widetilde{E_\mathbf{C}}+\widetilde{E_\mathbf{C}}
+\wedge^2\widetilde{E_\mathbf{C}}+(2\wedge^2\widetilde{E_\mathbf{C}}-\widetilde{E_\mathbf{C}}\otimes\widetilde{E_\mathbf{C}}\\
&+\widetilde{E_\mathbf{C}})\otimes(2\wedge^2\widetilde{L_{\bf{R}}\otimes\mathbf{C}}-(\widetilde{L_{\bf{R}}\otimes\mathbf{C}})\otimes(\widetilde{L_{\bf{R}}\otimes\mathbf{C}})+\widetilde{L_{\bf{R}}\otimes\mathbf{C}})+\wedge^2\widetilde{L_{\bf{R}}\otimes\mathbf{C}}\\
&\otimes\wedge^2\widetilde{L_{\bf{R}}
\otimes\mathbf{C}}+2\wedge^4\widetilde{L_{\bf{R}}\otimes\mathbf{C}}-2\widetilde{L_{\bf{R}}\otimes\mathbf{C}}\otimes\wedge^3\widetilde{L_{\bf{R}}
\otimes\mathbf{C}}+2\widetilde{L_{\bf{R}}\otimes\mathbf{C}}\otimes\wedge^2\widetilde{L_{\bf{R}}\otimes\mathbf{C}}\\
&-\widetilde{L_{\bf{R}}\otimes\mathbf{C}}\otimes\widetilde{L_{\bf{R}}
\otimes\mathbf{C}}\otimes\widetilde{L_{\bf{R}}
\otimes\mathbf{C}}+\widetilde{L_{\bf{R}}
\otimes\mathbf{C}}+\wedge^2\widetilde{L_{\bf{R}}\otimes\mathbf{C}})]\}^{(11)}=-196632\\
    &\cdot\{ e^{\frac{1}{720}(c_{2}(W_{i})+c_{2}(W_{j}))}{\widehat{A}(TM,\nabla^{TM})}{\rm exp}(\frac{c}{2}){\rm ch}(\triangle(E),g^{\triangle(E)},d) \}^{(11)}.\nonumber
\end{aligned}
\end{equation}
\end{thm}

\subsection{Modular forms and Witten genus in $(4k+1)$-dimensional spin$^c$ manifolds}\hfill\\

Set $M$ be a closed oriented ${\rm spin}^{c}$ manifold of dimension $4k+1.$
Consider
\begin{equation}
\begin{aligned}
\bar{Q}(\nabla^{TM},\nabla^{L},P_{i},P_{j},g,d,\tau)=&\{e^{\frac{1}{24}E_{2}(\tau)\cdot\frac{1}{30}(c_{2}(W_{i})+c_{2}(W_{j}))} \widehat{A}(TM,\nabla^{TM}){\rm exp}(\frac{c}{2})\\
&\cdot{\rm ch}(\Theta(T_{\mathbf{C}}M,L_{\bf{R}}\otimes\mathbf{C})){\rm ch}(Q(E),g^{Q(E)},d,\tau)\\
&\cdot\varphi(\tau)^{(16)}ch(\mathcal{V}_{i})ch(\mathcal{V}_{j})\}^{(4k+1)}.
\end{aligned}
\end{equation}

Applying a similar method leads to the following three new results.

\begin{thm}
Let ${\rm dim}M=4k+1,$ for $k\leq 3.$  If $p_1(L)-p_1(M)=0$ and $c_3(E,g,d)=0,$ then $\bar{Q}(\nabla^{TM},\nabla^{L},P_{i},P_{j},g,d,\tau)$ is a modular form over $SL(2,{\bf Z})$ with the weight $2k+8.$
\end{thm}

\begin{thm}
For $13$-dimensional spin manifold, when $p_1(L)-p_1(M)=0$ and $c_3(E,g,d)=0,$ we conclude that
 \begin{equation}
\begin{aligned}
&\{ {\widehat{A}(TM,\nabla^{TM})}{\rm exp}(\frac{c}{2}){\rm ch}(\triangle(E),g^{\triangle(E)},d)[-e^{\frac{1}{720}(c_{2}(W_{i})+c_{2}(W_{j}))}\cdot\tfrac{1}{30}(c_{2}(W_{i})\\
&+c_{2}(W_{j}))+e^{\frac{1}{720}(c_{2}(W_{i})+c_{2}(W_{j}))}{\rm ch}(W_i+W_j-16)+e^{\frac{1}{720}(c_{2}(W_{i})+c_{2}(W_{j}))}{\rm ch}(\widetilde{T_\mathbf{C}M}\\
&+\widetilde{E_\mathbf{C}}-\widetilde{E_\mathbf{C}} \otimes\widetilde{E_\mathbf{C}}+2\wedge^2\widetilde{E_\mathbf{C}}-\widetilde{L_{\bf{R}}\otimes\mathbf{C}})]\}^{(13)}=-24\{ e^{\frac{1}{720}(c_{2}(W_{i})+c_{2}(W_{j}))}\\
&{\widehat{A}(TM,\nabla^{TM})}{\rm exp}(\frac{c}{2}){\rm ch}(\triangle(E),g^{\triangle(E)},d) \}^{(13)},
\end{aligned}
\end{equation}
  and
 \begin{equation}
\begin{aligned}
&\{ {\widehat{A}(TM,\nabla^{TM})}{\rm exp}(\frac{c}{2}){\rm ch}(\triangle(E),g^{\triangle(E)},d)[e^{\frac{1}{720}(c_{2}(W_{i})+c_{2}(W_{j}))}\cdot\tfrac{1}{2}(-6+\tfrac{1}{30}(c_{2}(W_{i})\\
&+c_{2}(W_{j})))\cdot\tfrac{1}{30}(c_{2}(W_{i})+c_{2}(W_{j}))-e^{\frac{1}{720}(c_{2}(W_{i})+c_{2}(W_{j}))}\cdot\tfrac{1}{30}(c_{2}(W_{i})+c_{2}(W_{j}))\\
&\cdot{\rm ch}(W_i+W_j-16)+e^{\frac{1}{720}(c_{2}(W_{i})+c_{2}(W_{j}))}{\rm ch}(-16W_i-16W_j+W_iW_j+\overline{W_{i}}\\
&+\overline{W_{j}}+104)-e^{\frac{1}{720}(c_{2}(W_{i})+c_{2}(W_{j}))}\cdot\tfrac{1}{30}(c_{2}(W_{i})+c_{2}(W_{j})){\rm ch}(\widetilde{T_\mathbf{C}M}+\widetilde{E_\mathbf{C}}\end{aligned}
\end{equation}
\begin{equation}
\begin{aligned}
&-\widetilde{E_\mathbf{C}}
    \otimes\widetilde{E_\mathbf{C}}+2\wedge^2\widetilde{E_\mathbf{C}}-\widetilde{L_{\bf{R}}\otimes\mathbf{C}})+e^{\frac{1}{720}(c_{2}(W_{i})+c_{2}(W_{j}))}{\rm ch}(W_i+W_j-16)\\
    &\cdot{\rm ch}(\widetilde{T_\mathbf{C}M}+\widetilde{E_\mathbf{C}}-\widetilde{E_\mathbf{C}}
    \otimes\widetilde{E_\mathbf{C}}+2\wedge^2\widetilde{E_\mathbf{C}}-\widetilde{L_{\bf{R}}\otimes\mathbf{C}})+e^{\frac{1}{720}(c_{2}(W_{i})+c_{2}(W_{j}))}\\
    &\cdot{\rm ch}(\widetilde{T_\mathbf{C}M}+S^2\widetilde{T_\mathbf{C}M}+\widetilde{T_\mathbf{C}M}\otimes(\widetilde{E_\mathbf{C}}-\widetilde{E_\mathbf{C}}
    \otimes\widetilde{E_\mathbf{C}}+2\wedge^2\widetilde{E_\mathbf{C}}-\widetilde{L_{\bf{R}}\otimes\mathbf{C}})\\
    &-\widetilde{L_{\bf{R}}\otimes\mathbf{C}}+\wedge^2\widetilde{L_{\bf{R}}\otimes\mathbf{C}}-\widetilde{L_{\bf{R}}\otimes\mathbf{C}} \otimes(2\wedge^2\widetilde{E_\mathbf{C}}-\widetilde{E_\mathbf{C}}
    \otimes\widetilde{E_\mathbf{C}}+\widetilde{E_\mathbf{C}})+\wedge^2\widetilde{E_\mathbf{C}}\\
    &\otimes\wedge^2\widetilde{E_\mathbf{C}}
+2\wedge^4\widetilde{E_\mathbf{C}}-2\widetilde{E_\mathbf{C}}\otimes\wedge^3\widetilde{E_\mathbf{C}}
+2\widetilde{E_\mathbf{C}}\otimes\wedge^2\widetilde{E_\mathbf{C}}-\widetilde{E_\mathbf{C}}\otimes
\widetilde{E_\mathbf{C}}\otimes\widetilde{E_\mathbf{C}}\\
&+\widetilde{E_\mathbf{C}}+\wedge^2\widetilde{E_\mathbf{C}})]\}^{(13)}=-196632\{ e^{\frac{1}{720}(c_{2}(W_{i})+c_{2}(W_{j}))}{\widehat{A}(TM,\nabla^{TM})}{\rm exp}(\frac{c}{2})\\
    &\cdot{\rm ch}(\triangle(E),g^{\triangle(E)},d) \}^{(13)}.\nonumber
\end{aligned}
\end{equation}
\end{thm}

\section{Acknowledgements}

This research is supported by the Scientific Research Project of the Jilin Provincial Department of Education, China (Grant No. JJKH20261588KJ).

The authors also thank the referee for his (or her) careful reading and helpful comments.
The second author would like to thank Prof. Fei Han for helpful discussions.

\vskip 1 true cm

{\bf Data availability}. No data was gathered for this article.

{\bf Conflict of interest}. The authors have no relevant financial or non-financial interests to disclose.

\vskip 1 true cm


\bigskip
\bigskip

\noindent {\footnotesize {\it S. Liu} \\
{School of Mathematics and Statistics, Changchun University of Science and Technology, Changchun 130022, China}\\
{Email: liusy719@nenu.edu.cn}

\noindent {\footnotesize {\it Y. Wang} \\
{School of Mathematics and Statistics, Northeast Normal University, Changchun 130024, China}\\
{Email: wangy581@nenu.edu.cn}


\begin{thebibliography}{20}

\bibitem{AW} Alvarez-Gaum\'{e}, L., Witten E., Gravitational anomalies, {\it Nuclear Phys. B}, {\bf 234}, 1983, 269-330.

\bibitem{L} Liu K., Modular invariance and characteristic numbers, {\it Comm. Math. Phys.}, {\bf 174}, 1995, 29-42.

\bibitem{HZ1} Han F., Zhang W., Spin$^{c}$-manifold and elliptic genera, {\it C. R. Acad. Sci. Paris Ser. I.}, {\bf 336}, 2003, 1011-1014.

\bibitem{HZ2} Han F., Zhang W., Modular invariance, characteristic numbers and eta invariants, {\it J. Differential Geom.}, {\bf 67}, 2004, 257-288.

\bibitem{HH} Han F., Huang X., Even dimensional manifolds and generalized anomaly cancellation formulas, {\it Trans. Amer. Math. Soc.}, {\bf 359}, 2007, 5365-5382.

\bibitem{W1} Wang Y., A note on generalized twisted anomaly cancellation formulas, {\it Acta Math. Sin. Engl. Ser}, {\bf 26}, 2010, 1499-1508.

\bibitem{W2} Wang Y., Modular invariance and anomaly cancellation formulas, {\it Acta Math. Sin. Engl. Ser}, {\bf 7}, 2011, 1297-1304.

\bibitem{HLZ} Han F., Liu K., Zhang W., Modular forms and generalized anomaly cancellation formulas, {\it J. Geom. Phys}, {\bf 62}, 2012, 1038-1053.

\bibitem{LW} Liu K., Wang Y., A note on modular forms and generalized anomaly cancellation formulas, {\it Sci. China Math. Engl. Ser}, {\bf 56}, 2013, 55-65.

\bibitem{LW1} Liu K., Wang Y., Modular invariance and anomaly cancellation formulas in odd dimension, {\it J. Geom. Phys}, {\bf 99}, 2016, 190-200.

\bibitem{LW2} Liu K., Wang Y., Modular invariance and anomaly cancellation formulas in odd dimension II, {\it Acta Math. Sin. Engl. Ser}, {\bf 33}, 2017, 455-469.

\bibitem{L2} Liu K., On elliptic genera and theta-functions, {\it Topology}, {\bf 35}, 1996, 617-640.

\bibitem{CHZ} Chen Q., Han F., Zhang W., Generalized Witten genus and vanishing theorems, {\it J. Differential Geom.}, {\bf 88}, 2011, 1-40.

\bibitem{GWL} Guan J., Wang Y., Liu H., $SL(2,{\bf Z})$ modular forms and Witten genus in odd dimensions, To appear in J. Noncommut. Geom., https://doi.org/10.4171/JNCG/633

\bibitem{AE} Adams A., Evslin J., The loop group of $E_{8}$ and $K$-theory from $11d$, {\it J. High Energy Phys.}, {\bf 02}, 2003, 029.


\bibitem{HLZ2} Han F., Liu K., Zhang W., Anomaly cancellation and modularity, II: the $E_{8}\times E_{8}$ case, {\it Sci. China Math. Engl. Ser}, {\bf 60}, 2017, 985-996.

\bibitem{HHLZ} Han F., Huang R., Liu K., Zhang W., Cubic forms, anomaly cancellation and modularity, {\it Adv. Math.}, {\bf 394}, 2022, 108023.

\bibitem{WY} Wang Y., Yang Y., Anomaly cancellation formulas and $E_{8}$ bundles, To appear in Math. Res. Lett, https://doi.org/10.48550/arXiv.2305.00786

\bibitem{A} Atiyah M.F., K-theory, Addison-Wesley, California, 1967.

\bibitem{H} Hirzebruch F., Topological methods in algebraic geometry, Springer-Verlag, Berlin, 1966.

\bibitem{Z} Zhang W., Lectures on Chern-Weil theory and Witten deformations, World Scientific, Singapore, 2001.

\bibitem{C} Chandrasekharan K., Elliptic functions, Springer-Verlag, Berlin, 1985.

\bibitem{HY} Han F., Yu J., On the Witten rigidity theorem for odd dimensional manifolds,
https://doi.org/10.48550/arXiv.1504.03007




















\end{thebibliography}
\end{document}